\documentclass[12pt]{amsart}

\usepackage{amsmath, amssymb}
\usepackage[all]{xy}

\theoremstyle{plain}
\newtheorem{prop}{Proposition}[section]

\let\a\alpha
\let\b\beta

\let\l\lambda
\let\m\mu
\let\r\rho
\let\s\sigma

\let\f\varphi
\let\L\Lambda
\let\Si\Sigma
\let\Om\Omega

\def\lra{\longrightarrow}
\def\egal{\ar@{=}}

\def\A{\mathcal A}
\def\B{\mathcal B}
\def\C{\mathbb C}
\def\CC{\mathcal C}
\def\E{\mathcal E}
\def\F{\mathcal F}

\def\G{\mathcal G}
\def\HH{\mathcal H}
\def\I{\mathcal I}
\def\J{\mathcal J}
\def\K{\mathcal K}
\def\P{\mathbb P}
\def\S{\mathcal S}
\def\O{\mathcal O}
\def\U{\mathbb U}
\def\UU{\mathcal U}
\def\W{\mathbb W}

\def\Ker{{\mathcal Ker}}
\def\Coker{{\mathcal Coker}}
\def\Im{{\mathcal Im}}

\def\D{{\scriptscriptstyle \operatorname{D}}}
\def\T{{\scriptscriptstyle \operatorname{T}}}
\def\EE{\operatorname{E}}
\def\H{\operatorname{H}}
\def\h{\operatorname{h}}
\def\M{\operatorname{M}_{{\mathbb P}^2}}
\def\N{\operatorname{N}}
\def\pp{\operatorname{p}}
\def\PP{\operatorname{P}}
\def\SS{\operatorname{S}}
\def\Hom{\operatorname{Hom}}
\def\Aut{\operatorname{Aut}}
\def\Ext{\operatorname{Ext}}

\def\SL{\operatorname{SL}}
\def\rank{\operatorname{rank}}

\def\Grass{\operatorname{Grass}}

\def\tensor{\otimes}
\def\isom{\simeq}

\newcommand{\tilda}{\widetilde}
\newcommand{\noi}{\noindent}

\def\ds{\displaystyle}
\def\ba{\begin{array}}
\def\ea{\end{array}}

\begin{document}


\title[moduli of plane semi-stable sheaves with Hilbert polynomial $\operatorname{P}(m)=6m+1$]
{on the moduli space of semi-stable plane sheaves
with euler characteristic one and supported on sextic curves}

\author{mario maican}

\address{Mario Maican \\
Institute of Mathematics of the Romanian Academy \\
Calea Grivi\c tei 21 \\
 010702 Bucharest \\
Romania}

\email{mario.maican@imar.ro}

\begin{abstract}
We study the moduli space of Gieseker semi-stable sheaves on the complex
projective plane supported on sextic curves and having Euler characteristic
one. We determine locally free resolutions of length one for all such sheaves.
We decompose the moduli space into strata which occur naturally as quotients
modulo actions of certain algebraic groups. In some cases we give concrete
geometric descriptions of the strata.
\end{abstract}

\maketitle

\tableofcontents

\noi
{\sc Acknowledgements.} The author was supported by the
Consiliul Na\c tional al Cercet\u arii \c Stiin\c tifice,
an agency of the Romanian Government,
grant PN II--RU 169/2010 PD--219.

%%%%%%%%%%%%%%%%%%%%%%%%%%%% section 1

\section{Introduction}

\noi
This paper is concerned with the geometry of the moduli space $\M(6,1)$
of Gieseker semi-stable sheaves on $\P^2(\C)$ with Hilbert polynomial $\PP(m)=6m+1$,
i.e. semi-stable sheaves with support of dimension $1$, with multiplicity $6$ and Euler
characteristic $1$. The Fitting support of any such sheaf is a sextic curve.
This work is a continuation of \cite{drezet-maican} and \cite{mult_five},
where semi-stable sheaves supported on plane quartics, respectively plane quintics, were studied.
We refer to the introductory section of \cite{drezet-maican}
for a motivation of the problem and for a brief historical context.
We refer to the preliminaries section of op.cit. for a description of the techniques we shall use.

According to \cite{lepotier}, $\M(6,1)$ is a smooth irreducible rational projective variety of dimension $37$.
We shall decompose $\M(6,1)$ into five strata: an open stratum $X_0$,
two locally closed irreducible strata $X_1, X_2$ of codimensions $2$, respectively $4$,
a locally closed stratum that is the disjoint union of two irreducible locally closed subsets
$X_3$ and $X_4$, each of codimension $6$, and a closed irreducible stratum $X_5$ of codimension $8$.
The stratum $X_0$ is an open subset inside a fibre bundle with fibre $\P^{17}$ and base
the moduli space $\N(3,5,4)$ of semi-stable Kronecker modules $f \colon 5\O(-2) \to 4\O(-1)$;
$X_2$ is an open subset inside a fibre bundle with fibre $\P^{21}$ and base $Y \times \P^2$,
where $Y$ is the smooth projective variety of dimension $10$ constructed at \ref{4.2};
$X_3$ is an open subset inside a fibre bundle with fibre $\P^{23}$ and base $\P^2 \times \N(3,2,3)$,
where $\N(3,2,3)$ is the moduli space of semi-stable Kronecker modules $f \colon 2\O(-1) \to 3\O$;
$X_4$ is birational to a fibre bundle with base $\Grass(2,6)$ and fibre $\P^{23}$.
The closed stratum $X_5$ is isomorphic to the Hilbert flag scheme of sextic curves in $\P^2$
containing zero-dimensional subschemes of length $2$.

Each locally closed subset $X_i \subset \M(6,1)$ is defined by the cohomological conditions
listed in the second column of the table below.
We equip $X_i$ with the canonical induced reduced structure.
In the third column of the table we describe, by means of locally free resolutions of length $1$,
all semi-stable sheaves $\F$ on $\P^2$ whose stable-equivalence class is in $X_i$.
Thus, for each $X_i$ there are sheaves $\A_i, \B_i$ on $\P^2$, that are direct sums of line bundles,
such that each sheaf $\F$ giving a point in $X_i$ is the cokernel of some morphism $\f \in \Hom(\A_i, \B_i)$.
Let $W_i \subset \Hom(\A_i,\B_i)$ be the locally closed subset of injective morphisms $\f$ 
satisfying the conditions from the third column of the table below.
We equip $W_i$ with the canonical induced reduced structure.
In each case we shall prove that mapping $\f$ to $\Coker(\f)$ defines a map $W_i \to X_i$
that is a geometric quotient for the action by conjugation of $\Aut(\A_i) \times \Aut(\B_i)$.

Let $C$ be an arbitrary smooth sextic curve in $\P^2$.
The generic sheaves in $X_0$ have the form $\O_C(P_1+ \cdots + P_{10})$,
where $P_i$ are ten distinct points on $C$ not contained in a cubic curve.
The generic sheaves in $X_3$ have the form $\O_C(2)(-P_1-P_2-P_3+P_4)$,
where $P_i$ are four distinct points on $C$ and $P_1, P_2, P_3$ are non-colinear.
The generic sheaves in $X_4$ are of the form $\O_C(1)(P_1 +P_2 +P_3 +P_4)$,
where $P_i$ are four distinct points on $C$, no three of which are colinear.
The generic sheaves in $X_5$ are of the form $\O_C(2)(-P_1-P_2)$,
where $P_1, P_2$ are distinct points on $C$.

Let $\M(6,5)$ be the moduli space of semi-stable sheaves on $\P^2$
with Hilbert polynomial $\PP(m)= 6m+5$. According to \cite{maican-duality},
the map
\[
\F \lra \G = {\mathcal Ext}^1(\F, \omega_{\P^2}) \tensor \O(1)
\quad \text{gives an isomorphism} \quad
\M(6,1) \stackrel{\isom}{\lra} \M(6,5).
\]
Let $X_i^\D$ denote the image of $X_i$ under this isomorphism.
$\{ X_0^\D, X_1^\D, X_2^\D, X_3^\D \cup X_4^{\D}, X_5^\D \}$
represents a stratification of $\M(6,5)$.
$X_i^{\D}$ is defined by the dual cohomological conditions,
e.g. $X_3^{\D}$ is defined by the conditions
\[
\h^0(\G(-1))=2, \qquad
\h^1(\G) = 0, \qquad
\h^1(\G \tensor \Om^1(1))=2.
\]
According to op.cit., lemma 3, the sheaves $\G$ on $\P^2$ giving points in $X_i^{\D}$
are precisely the sheaves of the form $\Coker({\mathcal Hom}(\f, \omega_{\P^2})) \tensor \O(1)$,
$\f \in W_i$.
For instance, $X_3^\D$ consists of the stable-equivalence classes of those sheaves $\G$
having resolution of the form
\[
0 \lra 3\O(-2) \oplus \O \stackrel{\psi}{\lra} 2\O(-1) \oplus 2\O(1) \lra \G \lra 0,
\]
where $\psi_{11}$ has linearly independent maximal minors and $\psi_{22}$ has linearly
independent entries. In this fashion we can obtain a ``dual'' table describing all semi-stable sheaves
on $\P^2$ with Hilbert polynomial $\PP(m)=6m+5$, the conditions on the morphisms being
the transposed conditions from the table below.
We omit the details.

\begin{table}[!hpt]{}
\begin{center}
{\small
\begin{tabular}{|c|c|c|}
\hline \hline
{}
&
\begin{tabular}{c}
{\tiny cohomological} \\
{\tiny conditions}
\end{tabular}
&
$W$
\\
\hline
$X_0$
&
\begin{tabular}{r}
{} \\
$\h^0(\F(-1))=0$ \\
$\h^1(\F)=0$\\
$\h^0(\F \tensor \Om^1(1))=0$ \\
{}
\end{tabular}
&
\begin{tabular}{c}
$0 \lra 5\O(-2) \stackrel{\f}{\lra} 4\O(-1) \oplus \O \lra \F \lra 0$ \\
$\f_{11}$ is semi-stable as a Kronecker module
\end{tabular} \\
\hline
$X_1$
&
\begin{tabular}{r}
$\h^0(\F(-1))=0$ \\
$\h^1(\F)=1$\\
$\h^0(\F \tensor \Om^1(1))=0$
\end{tabular}
&
\begin{tabular}{c}
{} \\
$0 \lra \O(-3) \oplus 2\O(-2) \stackrel{\f}{\lra} \O(-1) \oplus 2\O \lra \F \lra 0$ \\
$\f$ is not equivalent to a morphism of the form \\
${\ds
\left[
\ba{ccc}
\star & 0 & 0 \\
\star & \star & \star \\
\star & \star & \star
\ea
\right], \left[
\ba{ccc}
\star & \star & 0 \\
\star & \star & 0 \\
\star & \star & \star
\ea
\right], \left[
\ba{ccc}
\star & \star & \star \\
\star & \star & \star \\
\star & 0 & 0
\ea
\right], \left[
\ba{ccc}
0 & 0 & \star \\
\star & \star & \star \\
\star & \star & \star
\ea
\right]
}$ \\
{}
\end{tabular} \\
\hline
$X_2$
&
\begin{tabular}{r}
$\h^0(\F(-1))=0$ \\
$\h^1(\F)=1$\\
$\h^0(\F \tensor \Om^1(1))=1$
\end{tabular}
&
\begin{tabular}{c}
{} \\
$0 \to \O(-3) \oplus 2\O(-2) \oplus \O(-1) \stackrel{\f}{\to} 2\O(-1) \oplus 2\O \to \F \to 0$ \\
${\ds \f = \left[
\ba{cccc}
q_1 & \ell_{11} & \ell_{12} & 0 \\
q_2 & \ell_{21} & \ell_{22} & 0 \\
f_1 & q_{11} & q_{12} & \ell_1 \\
f_2 & q_{21} & q_{22} & \ell_2
\ea
\right]}$ \\
$\ell_1, \ell_2$ are linearly independent, $\ell_{11} \ell_{22} - \ell_{12} \ell_{21} \neq 0$, \\
${\ds \left|
\ba{cc}
q_1 & \ell_{11} \\
q_2 & \ell_{21}
\ea
\right|}$, ${\ds \left|
\ba{cc}
q_1 & \ell_{12} \\
q_2 & \ell_{22}
\ea
\right| }$ are linearly indep. modulo
${\ds \left|
\ba{cc}
\ell_{11} & \ell_{12} \\
\ell_{21} & \ell_{22}
\ea
\right|}$ \\
{}
\end{tabular} \\
\hline
$X_{3}$
&
\begin{tabular}{r}
$\h^0(\F(-1))=0$ \\
$\h^1(\F)=2$\\
$\h^0(\F \tensor \Om^1(1))=2$
\end{tabular}
&
\begin{tabular}{c}
{} \\
$0 \lra 2\O(-3) \oplus 2\O(-1) \stackrel{\f}{\lra} \O (-2) \oplus 3\O \lra \F \lra 0$ \\
$\f_{11}$ has linearly independent entries \\
$\f_{22}$ has linearly independent maximal minors \\
{}
\end{tabular} \\
\hline
$X_{4}$
&
\begin{tabular}{r}
$\h^0(\F(-1))=1$ \\
$\h^1(\F)=2$\\
$\h^0(\F \tensor \Om^1(1))=3$
\end{tabular}
&
\begin{tabular}{c}
{} \\
$0 \to 2\O(-3) \oplus \O(-2) \stackrel{\f}{\to} \O (-2) \oplus \O(-1) \oplus \O(1) \to \F \to 0$ \\
${\ds \f= \left[
\ba{ccc}
0 & 0 & 1 \\
q_1 & q_2 & 0 \\
g_1 & g_2 & 0
\ea
\right], }$ \\
where $q_1, q_2$ have no common factor or \\
${\ds \f= \left[
\ba{ccc}
\ell_1 & \ell_2 & 0 \\
q_1 & q_2 & \ell \\
g_1 & g_2 & h
\ea
\right] }$, \\
where $\ell_1, \ell_2$ are linearly independent, $\ell \neq 0$ \\
and $\f$ is not equivalent to a morphism of the form \\
${\ds \left[
\ba{ccc}
\star & \star & 0 \\
0 & 0 & \star \\
\star & \star & \star
\ea
\right]}$ \\
{}
\end{tabular} \\
\hline
$X_{5}$
&
\begin{tabular}{r}
$\h^0(\F(-1))=1$ \\
$\h^1(\F)=3$\\
$\h^0(\F \tensor \Om^1(1))=4$
\end{tabular}
&
\begin{tabular}{c}
{} \\
$0 \lra \O(-4) \oplus \O(-1) \stackrel{\f}{\lra} \O \oplus \O(1) \lra \F \lra 0$ \\
${\ds \f = \left[
\ba{cc}
h & \ell \\
g & q
\ea
\right]}$, where $\ell \neq 0$ and $\ell$ does not divide $q$ \\
{}
\end{tabular} \\
\hline \hline
\end{tabular}
}
\end{center}
\end{table}

\newpage

\noi
{\sc Notations.} $V$ is a three-dimensional vector space over $\C$; we identify $\P(V)$ with $\P^2$;
$\{ X, Y, Z \}$ is a basis of $V^*$; $\F^\D = {\mathcal Ext}^1(\F,\omega_{\P^2})$ is the dual
of a one-dimensional sheaf $\F$ on $\P^2$; $\M(r,\chi)$ denotes the moduli space of semi-stable
sheaves $\F$ on $\P^2$ with Hilbert polynomial $\PP(m)=rm +\chi$;
$\pp(\F)= \chi/r$ is the slope of $\F$; $\N(q,m,n)$ is the moduli space of Kronecker modules
$\tau \colon \C^m \tensor \C^q \to \C^n$ that are semi-stable for the action of $\SL(m) \times \SL(n)$
(see 2.4 \cite{drezet-maican}).
For any other unexplained notations we refer to \cite{drezet-maican} and \cite{mult_five}.

%%%%%%%%%%%%%%%%%%%%%%%%%%%% section 2

\section{The open stratum}

\begin{prop}
\label{2.1}
Every sheaf $\F$ giving a point in $\M(6,1)$ and satisfying the condition
$\h^1(\F)=0$ also satisfies the condition $\h^0(\F(-1))=0$.
These sheaves are precisely the sheaves with resolution of the form
\[
0 \lra 5\O(-2) \stackrel{\f}{\lra} 4\O(-1) \oplus \O \lra \F \lra 0,
\]
where $\f$ is not equivalent to a morphism of the form

\noi
\hfill ${\ds
\left[
\ba{cc}
\psi & 0 \\
\star & \star
\ea
\right]}$,
with $\psi \colon m\O(-2) \to m\O(-1)$, $m=1, 2, 3, 4$.
\end{prop}

\begin{proof}
The statement follows by duality from 4.2 \cite{maican}.
\end{proof}

\noi
Let $\W_0 = \Hom(5\O(-2), 4\O(-1) \oplus \O)$ and let $W_0 \subset \W_0$ be the set of morphisms
$\f$ from the proposition above. Let $G_0= (\Aut(5\O(-2)) \times \Aut(4\O(-1) \oplus \O))/\C^*$
be the natural group acting by conjugation on $\W_0$. Let $X_0 \subset \M(6,1)$ be the set of
stable-equivalence classes of sheaves $\F$ as in the proposition above. Note that $X_0$ is open and dense.

\begin{prop}
\label{2.2}
There exists a geometric quotient $W_0/G_0$ and it is a proper open subset inside
a fibre bundle with fibre $\P^{17}$ and base $\N(3,5,4)$.
Moreover, $W_0/G_0$ is isomorphic to $X_0$.
\end{prop}

\begin{proof}
The argument is identical to the argument at 3.2.1 \cite{mult_five}.
Let $\L = (\l_1,\m_1,\m_2)$ be a polarisation for the action of $G_0$ on $\W_0$
satisfying $0 < \m_2 < 1/5$; $W_0$ is the proper open invariant subset of injective
morphisms inside $\W_0^{ss}(\L)$. Let $\N(3,5,4)$ be the moduli space of semi-stable
Kronecker modules $f \colon 5\O(-2) \to 4\O(-1)$ and let
\[
\theta \colon p_1^*(E) \tensor p_2^*(\O(-2)) \lra p_1^*(F) \tensor p_2^*(\O(-1))
\]
be the morphism of sheaves on $\N(3,5,4) \times \P^2$ induced from the universal morphism
(notations as at 3.1.2 \cite{drezet-maican}).
Then $\UU={p_1}_{*}(\Coker(\theta^*))$ is a vector bundle of rank $18$ on $\N(3,5,4)$
and $\P(\UU)$ is the geometric quotient $\W_0^{ss}(\L)/G_0$.
Thus $W_0/G_0$ exists and is a proper open subset of $\P(\UU)$.

The canonical morphism $W_0/G_0 \to X_0$ is bijective and,
since $X_0$ is smooth, it is an isomorphism.
\end{proof}

\noi
Let $Y_0 \subset X_0$ be the open subset of points given by sheaves $\F = \Coker(\f)$
for which the maximal minors of $\f_{11}$ have no common factor.
Let $Y_0^\D \subset \M(6,5)$ be the dual subset.

\begin{prop}
\label{2.3}
The sheaves $\G$ from $Y_0^\D$ have the form $\J_Z(4)$,
where $Z \subset \P^2$ is a zero-dimensional scheme of length $10$ not contained
in a cubic curve, contained in a sextic curve $C$, and $\J_Z \subset \O_C$ is its ideal sheaf.

The generic sheaves $\G$ in $X_0^\D$ have the form $\O_C(4)(-P_1- \cdots - P_{10})$,
where $C \subset \P^2$ is a smooth sextic curve and $P_i$, $1 \le i \le 10$, are distinct points on $C$
not contained in a cubic curve. By duality, the generic sheaves $\F$ in $X_0$ have the form
$\O_C(P_1 + \cdots + P_{10})$.
\end{prop}

\begin{proof}
The sheaves $\G$ from $Y_0^\D$ are precisely the sheaves having resolution
\[
0 \lra \O(-2) \oplus 4\O(-1) \stackrel{\psi}{\lra} 5\O \lra \G \lra 0,
\]
where the maximal minors of $\psi_{12}$ have no common factor.
In particular, $\psi_{12}$ is semi-stable as a Kronecker $V$-module.
According to \cite{modules-alternatives}, propositions 4.5 and 4.6, $\Coker(\psi_{12}) \isom \I_Z(4)$,
where $Z \subset \P^2$ is a zero-dimensional scheme of length $10$ not contained
in a cubic curve. Thus $\G \isom \J_Z(4)$, where $C$ is given by the equation $\det(\psi)=0$.
The converse is clear.
\end{proof}

%%%%%%%%%%%%%%%%%%%%%%%%%%%%%%%%%%%% section 3

\section{The codimension $2$ stratum}

\begin{prop}
\label{3.1}
Let $\F$ be a sheaf giving a point in $\M(6,1)$ and satisfying the conditions
$\h^1(\F) = 1$ and $\h^0(\F(-1))=0$. Then $\h^0(\F \tensor \Om^1(1))= 0$ or $1$.
The sheaves in the first case are precisely the sheaves with resolution of the form
\[
0 \lra \O(-3) \oplus 2\O(-2) \stackrel{\f}{\lra} \O(-1) \oplus 2\O \lra \F \lra 0,
\]
\[
\f= \left[
\ba{ccc}
q & \ell_1 & \ell_2 \\
f_1 & q_{11} & q_{12} \\
f_2 & q_{21} & q_{22}
\ea
\right],
\]
where $\f$ is not equivalent to a morphism represented by a matrix of one of the
following four forms:
\[
\f_1 = \left[
\ba{ccc}
\star & 0 & 0 \\
\star & \star & \star \\
\star & \star & \star
\ea
\right], \quad \f_2= \left[
\ba{ccc}
\star & \star & 0 \\
\star & \star & 0 \\
\star & \star & \star
\ea
\right], \quad \f_3 = \left[
\ba{ccc}
\star & \star & \star \\
\star & \star & \star \\
\star & 0 & 0
\ea
\right], \quad \f_4 = \left[
\ba{ccc}
0 & 0 & \star \\
\star & \star & \star \\
\star & \star & \star
\ea
\right].
\]
\end{prop}

\begin{proof}
Assume that $\F$ gives a point in $\M(6,1)$ and satisfies the conditions
$\h^1(\F)=1$, $\h^0(\F(-1))=0$. Write $m= \h^0(\F \tensor \Om^1(1))$.
The Beilinson free monad for $\F$
\[
0 \lra 5\O(-2) \oplus m\O(-1) \lra (m+4)\O(-1) \oplus 2\O \lra \O \lra 0
\]
gives the resolution
\[
0 \lra 5\O(-2) \oplus m\O(-1) \lra \Om^1 \oplus (m+1) \O(-1) \oplus 2\O \lra \F \lra 0,
\]
hence the exact sequence
\[
0 \lra \O(-3) \oplus 5\O(-2) \oplus m\O(-1) \stackrel{\f}{\lra} 3\O(-2) \oplus (m+1)\O(-1) \oplus 2\O
\lra \F \lra 0
\]
with $\f_{13}=0$, $\f_{23}=0$. As in the proof of 2.1.4 \cite{mult_five}, we have
$\rank(\f_{12})=3$, hence we have the resolution
\[
0 \lra \O(-3) \oplus 2\O(-2) \oplus m\O(-1) \stackrel{\f}{\lra} (m+1)\O(-1) \oplus 2\O
\lra \F \lra 0
\]
with $\f_{13}=0$. From the injectivity of $\f$ we see that $m \le 2$. If $m=2$, then
$\Coker(\f_{23})$ would be a destabilising subsheaf of $\F$.
We conclude that $m=0$ or $1$.
In the first case $\F$ has resolution as in the proposition.
The conditions on $\f$ follow from the semi-stability of $\F$.

\medskip

\noi
Conversely, we assume that $\F$ has resolution as in the proposition and we need to show that there are no
destabilising subsheaves $\E$. We argue by contradiction, i.e. we assume that there is such
a subsheaf $\E$. We may take $\E$ to be semi-stable.
As $\h^0(\E) \le 2$, $\E$ gives a point in $\M(r,1)$ or $\M(r,2)$ for some $r$, $1 \le r \le 5$.
The cohomology groups $\H^0(\E(-1))$ and $\H^0(\E \tensor \Om^1(1))$ vanish
because the corresponding cohomology groups for $\F$ vanish.
From the description of $\M(r,1)$ and $\M(r,2)$, $1 \le r \le 5$, found in \cite{drezet-maican}
and \cite{mult_five}, we see that $\E$ may have one of the following resolutions:
\[
\tag{1}
0 \lra \O(-2) \lra \O  \lra \E \lra 0,
\]
\[
\tag{2}
0 \lra 2\O(-2) \lra \O(-1) \oplus \O \lra \E \lra 0,
\]
\[
\tag{3}
0 \lra 3\O(-2) \lra 2\O(-1) \oplus \O \lra \E \lra 0,
\]
\[
\tag{4}
0 \lra 2\O(-2) \lra 2\O \lra \E \lra 0,
\]
\[
\tag{5}
0 \lra 4\O(-2) \lra 3\O(-1) \oplus \O \lra \E \lra 0,
\]
\[
\tag{6}
0 \lra \O(-3) \oplus \O(-2) \lra 2\O \lra \E \lra 0,
\]
\[
\tag{7}
0 \lra 3\O(-2) \lra \O(-1) \oplus 2\O \lra \E \lra 0.
\]
Resolution (1) must fit into a commutative diagram
\[
\tag{*}
\xymatrix
{
0 \ar[r] & \O(-2) \ar[r]^-{\psi} \ar[d]^-{\b} & \O \ar[r] \ar[d]^-{\a} & \E \ar[r] \ar[d] & 0 \\
0 \ar[r] & \O(-3) \oplus 2\O(-2) \ar[r]^-{\f} & \O(-1) \oplus 2\O \ar[r] & \F \ar[r] & 0
}
\]
in which $\a$ is injective (being injective on global sections).
Thus $\b$ is injective, too, and $\f \sim \f_2$, contradicting our hypothesis on $\f$.
Similarly, every other resolution must fit into a commutative diagram in which
$\a$ and $\a(1)$ are injective on global sections.
This rules out resolution (7) because in that case $\a$ must be injective,
hence $\Ker(\b)=0$, which is absurd.
If $\E$ has resolution (5), then $\a$ is equivalent to a morphism
represented by a matrix having one of the following two forms:
\[
\left[
\ba{cccc}
1 & 0 & 0 & 0 \\
0 & u_1 & u_2 & 0 \\
0 & 0 & 0 & 1
\ea
\right] \quad \text{or} \quad \left[
\ba{cccc}
0 & 0 & 0 & 0 \\
u_1 & u_2 & u_3 & 0 \\
0 & 0 & 0 & 1
\ea
\right],
\]
where $u_1, u_2, u_3$ are linearly independent one-forms.
In the first case $\Ker(\b) \isom \O(-2)$, in the second case $\Ker(\b) \isom \Om^1$.
Both situations are absurd.
Assume that $\E$ has resolution (3).
Since $\b$ cannot be injective, we see that $\a$ is equivalent to a morphism
represented by a matrix of the form
\[
\left[
\ba{ccc}
0 & 0 & 0 \\
u_1 & u_2 & 0 \\
0 & 0 & 1
\ea
\right],
\]
hence $\Ker(\a) \isom \O(-2)$, hence $\f \sim \f_1$, which is a contradiction.
For resolutions (2), (4) and (6) $\a$ and $\b$ must be injective and we get
the contradictions $\f \sim \f_3$, $\f \sim \f_1$, or $\f \sim \f_4$.
\end{proof}

\noi
Let $\W_1 = \Hom(\O(-3) \oplus 2\O(-2), \O(-1) \oplus 2\O)$ and let $W_1 \subset \W_1$
be the set of morphisms $\f$ from proposition \ref{3.1}. Let
\[
G_1 = (\Aut(\O(-3) \oplus 2\O(-2)) \times \Aut(\O(-1) \oplus 2\O))/\C^*
\]
be the natural group acting by conjugation on $\W_1$.
Let $X_1 \subset \M(6,1)$ be the set of isomorphism classes of sheaves of the form $\Coker(\f)$,
$\f \in W_1$. Note that $X_1$ is a locally closed subset; we equip it with the canonical induced
reduced structure.

\begin{prop}
\label{3.2}
There exists a geometric quotient $W_1/G_1$ and it is isomorphic to $X_1$.
\end{prop}

\begin{proof}
We have a canonical morphism $\r \colon W_1 \to X_1$ mapping $\f$ to the isomorphism
class of $\Coker(\f)$. Its fibres are $G_1$-orbits. In order to show that $\r$ is a categorical
quotient map we use the method of 3.1.6 \cite{drezet-maican}. We need to show that for
given $\F$ in $X_1$ resolution \ref{3.1} can be obtained in a natural manner from the Beilinson
spectral sequence converging to $\F$.
We prefer, instead, to start with the Beilinson spectral sequence of the dual sheaf
$\G = \F^\D(1)$, which gives a point in $\M(6,5)$.
Diagram (2.2.3) \cite{drezet-maican} for $\G$ takes the form
\[
\xymatrix
{
2\O(-2) & 0 & 0 \\
\O(-2) \ar[r]^-{\f_3} & 4\O(-1) \ar[r]^-{\f_4} & 5\O
}.
\]
The exact sequence (2.2.5) \cite{drezet-maican} reads:
\[
0 \lra 2\O(-2) \stackrel{\f_5}{\lra} \Coker(\f_4) \lra \G \lra 0.
\]
We see from this that $\Coker(\f_4)$ has no zero-dimensional torsion.
The exact sequence (2.2.4) \cite{drezet-maican} takes the form
\[
0 \lra \O(-2) \stackrel{\f_3}{\lra} 4\O(-1) \stackrel{\f_4}{\lra} 5\O \lra \Coker(\f_4) \lra 0.
\]
As in the proof of 3.2.4 \cite{mult_five}, we can show that $\Coker(\f_3)$ is isomorphic to
$\O(-1) \oplus \Om^1(1)$. The argument, we recall, uses the fact that $\Coker(\f_4)$
has no zero-dimensional torsion. Thus we have an exact sequence
\[
0 \lra \O(-1) \oplus \Om^1(1) \lra 5\O \lra \Coker(\f_4) \lra 0.
\]
Combining with the above resolution of $\G$ yields the resolution
\[
0 \lra 2\O(-2) \oplus \O(-1) \oplus \Om^1(1) \lra 5\O \lra \G \lra 0.
\]
Dualising we obtain the exact sequence
\[
0 \lra 5\O(-2) \lra \Om^1 \oplus \O(-1) \oplus 2\O \lra \F \lra 0.
\]
We saw in the proof of proposition \ref{3.1} how this leads to the desired resolution of $\F$.

Thus far we have proved that $\rho$ is a categorical quotient map.
According to \cite{mumford}, remark (2), p. 5, $X_1$ is normal.
Applying \cite{popov-vinberg}, theorem 4.2, we conclude that $\rho$ is a geometric quotient
map.\footnote{The author is grateful to J.-M. Dr\'ezet for pointing out this reference.}
\end{proof}

%%%%%%%%%%%%%%%%%%%%%%%%%%%%%%%%%%%%%% section 4

\section{The codimension $4$ stratum}

\begin{prop}
\label{4.1}
The sheaves $\F$ giving points in $\M(6,1)$ and satisfying the cohomological conditions
\[
\h^1(\F) = 1, \qquad \h^0(\F(-1))=0, \qquad \h^0(\F \tensor \Om^1(1))= 1
\]
are precisely the sheaves with resolution of the form
\[
0 \lra \O(-3) \oplus 2\O(-2) \oplus \O(-1) \stackrel{\f}{\lra} 2\O(-1) \oplus 2\O \lra \F \lra 0,
\]
\[
\f = \left[
\ba{cccc}
q_1 & \ell_{11} & \ell_{12} & 0 \\
q_2 & \ell_{21} & \ell_{22} & 0 \\
f_1 & q_{11} & q_{12} & \ell_1 \\
f_2 & q_{21} & q_{22} & \ell_2
\ea
\right],
\]
where $\ell_1, \ell_2$ are linearly independent one-forms,
$\ell_{11} \ell_{22} - \ell_{12} \ell_{21} \neq 0$ and the images of

\noi
\hfill ${\ds \left|
\ba{cc}
q_1 & \ell_{11} \\
q_2 & \ell_{21}
\ea
\right|}$
and 
${\ds \left|
\ba{cc}
q_1 & \ell_{12} \\
q_2 & \ell_{22}
\ea
\right|}$
in $\SS^3 V^*/(\ell_{11} \ell_{22} - \ell_{12} \ell_{21})V^*$ are linearly independent.

Notice that the last two conditions are equivalent
to saying that $\f$ is not equivalent to a morphism represented by a matrix of one of the
following four forms:
\[
\f_1 = \left[
\ba{cccc}
\star & 0 & 0 & 0 \\
\star & \star & \star & 0 \\
\star & \star & \star & \star \\
\star & \star & \star & \star
\ea
\right]\!\! , \ \f_2 = \left[
\ba{cccc}
\star & \star & 0 & 0 \\
\star & \star & 0 & 0 \\
\star & \star & \star & \star \\
\star & \star & \star & \star
\ea
\right]\!\! , \ \f_3 = \left[
\ba{cccc}
0 & \star & 0 & 0 \\
\star & \star & \star & 0 \\
\star & \star & \star & \star \\
\star & \star & \star & \star
\ea
\right] \!\! , \ \f_4 = \left[
\ba{cccc}
0 & \star & \star & 0 \\
0 & \star & \star & 0 \\
\star & \star & \star & \star \\
\star & \star & \star & \star
\ea
\right].
\]
\end{prop}

\begin{proof}
At \ref{3.1} we proved that a sheaf $\F$ in $\M(6,1)$ satisfying the above cohomologial
conditions has a resolution as in the proposition. The conditions on $\f$ follow from the semi-stability
of $\F$.

Assume now that $\F$ has a resolution as in the proposition and that $\E$ is a destabilising subsheaf.
We have $\h^0(\E(-1))=0$, $\h^0(\E \tensor \Om^1(1)) \le 1$ and, as at \ref{3.1}, we may assume that
$\E$ gives a point in $\M(r,1)$ or $\M(r,2)$ for some $r$, $1 \le r \le 5$.
From the above we see that $\E$ may have one of the following resolutions:
\[
\tag{1}
0 \lra  \O(-1) \lra \O \lra  \E \lra 0,
\]
\[
\tag{2}
0 \lra \O(-2) \lra \O \lra \E \lra 0,
\]
\[
\tag{3}
0 \lra 2\O(-2) \lra \O(-1) \oplus \O \lra \E \lra 0,
\]
\[
\tag{4}
0 \lra \O(-2) \oplus \O(-1) \lra 2\O  \lra \E \lra 0,
\]
\[
\tag{5}
0 \lra 3\O(-2) \lra 2\O(-1) \oplus \O \lra \E \lra 0,
\]
\[
\tag{6}
0 \lra \O(-3) \oplus \O(-1) \lra 2\O \lra \E \lra 0,
\]
\[
\tag{7}
0 \lra 2\O(-2) \lra 2\O \lra \E \lra 0,
\]
\[
\tag{8}
0 \lra 2\O(-2) \oplus \O(-1) \lra \O(-1) \oplus 2\O \lra \E \lra 0,
\]
\[
\tag{9}
0 \lra 4\O(-2) \lra 3\O(-1) \oplus \O \lra \E \lra 0,
\]
\[
\tag{10}
0 \lra \O(-3) \oplus \O(-2) \lra 2\O \lra \E \lra 0,
\]
\[
\tag{11}
0 \lra \O(-3) \oplus \O(-2) \oplus \O(-1) \lra \O(-1) \oplus 2\O \lra \E \lra 0,
\]
\[
\tag{12}
0 \lra 3\O(-2) \lra \O(-1) \oplus 2\O \lra \E \lra 0,
\]
\[
\tag{13}
0 \lra 3\O(-2) \oplus \O(-1) \lra 2\O(-1) \oplus 2\O \lra \E \lra 0.
\]
Resolutions (1), (2), (4), (6), (7) and (10) must fit into commutative diagrams as diagram (*) at \ref{3.1}
in which $\a$ is injective, being injective on global sections.
Thus $\b$ is also injective and we obtain the contradictory conclusions that
$\ell_1, \ell_2$ are linearly dependent or that $\f$ is equivalent to $\f_1$, $\f_2$ or $\f_4$.
The remaining resolutions also fit into commutative diagrams (*) in which $\a$ is injective on global
sections and $\a(1)$ restricted to the first direct summand is also injective on global sections.
Thus $\a$ is injective in the case of resolutions (3), (8), (11), (12) and (13).
This rules out resolution (13) since in that case $\b$ cannot be injective.
Assume that $\E$ has resolution (3). Then
\[
\a \sim \left[
\ba{cc}
0 & 0 \\
0 & 0 \\
u & 0 \\
0 & 1
\ea
\right] \quad \text{or} \quad \a \sim \left[
\ba{cc}
0 & 0 \\
1 & 0 \\
0 & 1 \\
0 & 0
\ea
\right] \quad \text{and} \quad \b \sim \left[
\ba{cc}
0 & 0 \\
1 & 0 \\
0 & 1 \\
0 & 0
\ea
\right] \quad \text{or} \quad \b \sim \left[
\ba{cc}
0 & 0 \\
0 & 0 \\
1 & 0 \\
0 & v
\ea
\right]
\]
for some non-zero $u, v \in V^*$. We obtain the contradictory conclusions that
$\f \sim \f_1$, or $\f \sim \f_2$, or that $\ell_1$ and $\ell_2$ are linearly dependent.
For resolutions (8), (11) and (12) we have
\[
\a \sim \left[
\ba{ccc}
0 & 0 & 0 \\
1 & 0 & 0 \\
0 & 1 & 0 \\
0 & 0 & 1
\ea
\right] \quad \text{and} \quad \b \sim \left[
\ba{ccc}
0 & 0 & 0 \\
\star & 0 & 0 \\
0 & \star & 0 \\
0 & 0 & \star
\ea
\right] \quad \text{or} \quad \b \sim \left[
\ba{ccc}
\star & 0 & 0 \\
0 & 0 & 0 \\
0 & \star & 0 \\
0 & 0 & \star
\ea
\right],
\]
where ``$\star$'' stands for a non-zero entry.
We obtain the contradictory conclusions that $\f \sim \f_1$ or $\f \sim \f_3$.
Assume that $\E$ has resolution (5). Then $\a$ is equivalent to a morphism
represented by a matrix having one of the following forms:
\[
\left[
\ba{ccc}
1 & 0 & 0 \\
0 & 1 & 0 \\
0 & 0 & 1 \\
0 & 0 & 0
\ea
\right], \qquad \left[
\ba{ccc}
0 & 0 & 0 \\
1 & 0 & 0 \\
0 & u & 0 \\
0 & 0 & 1
\ea
\right], \qquad \left[
\ba{ccc}
0 & 0 & 0 \\
0 & 0 & 0 \\
u_1 & u_2 & 0 \\
0 & 0 & 1
\ea
\right],
\]
where $u \neq 0$ and $u_1, u_2$ are linearly independent one-forms.
In the first two cases $\b$ is injective, so it has the form
\[
\left[
\ba{ccc}
0 & 0 & 0 \\
1& 0 & 0 \\
0 & 1 & 0 \\
0 & 0 & v
\ea
\right]
\]
for some non-zero $v \in V^*$. We obtain the contradictory conclusions that $\ell_1, \ell_2$
are linearly dependent or that $\f \sim \f_1$.
In the third case we have $\Ker(\b) \isom \O(-2)$, hence $\b$ has one of the following two forms:
\[
\left[
\ba{ccc}
0 & 0 & 0 \\
0 & 1 & 0 \\
0 & 0 & 1 \\
0 & 0 & 0
\ea
\right] \qquad \text{or} \qquad \left[
\ba{ccc}
0 & 0 & 0 \\
0 & 0 & 0 \\
0 & 1 & 0 \\
0 & 0 & v
\ea
\right]
\]
for some non-zero $v \in V^*$.
We get $\f \sim \f_1$ or $\f \sim \f_2$, both contradictions.
Finally, assume that $\E$ has resolution (9).
Notice that $\b$, hence also $\a$, cannot be injective.
As $\a$ and $\a(1)$ are injective on global sections, we deduce
that $\a$ is equivalent to a morphism represented by a matrix having
one of the following forms:
\[
\left[
\ba{cccc}
0 & 0 & 0 & 0 \\
1 & 0 & 0 & 0 \\
0 & u_1 & u_2 & 0 \\
0 & 0 & 0 & 1
\ea
\right] \qquad \text{or} \qquad \left[
\ba{cccc}
0 & 0 & 0 & 0 \\
0 & 0 & 0 & 0 \\
u_1 & u_2 & u_3 & 0 \\
0 & 0 & 0 & 1
\ea
\right],
\]
where $u_1, u_2, u_3$ are linearly independent one-forms.
In the first case we have $\Ker(\b) \isom \O(-2)$, hence $\b$ has the form
\[
\left[
\ba{cccc}
0 & 0 & 0 & 0 \\
0 & 1 & 0 & 0 \\
0 & 0 & 1 & 0 \\
0 & 0 & 0 & v
\ea
\right]
\]
for some non-zero $v \in V^*$. We obtain $\f \sim \f_1$, which contradicts our hypothesis on $\f$.
In the second case we have $\Ker(\b) \isom \Om^1$.
This is absurd, $\Om^1$ cannot be a subsheaf of $4\O(-2)$.
\end{proof}

\noi
Consider the vector space $\U = \Hom(\O(-3) \oplus 2\O(-2), 2\O(-1))$ which is acted upon
by the algebraic group $G = (\Aut(\O(-3) \oplus 2\O(-2)) \times \Aut(2\O(-1)))/\C^*$.
We represent the elements of $G$ by pairs $(g,h)$ of matrices
\[
g = \left[
\ba{ccc}
g_{11} & 0 & 0 \\
u_{21} & g_{22} & g_{23} \\
u_{31} & g_{32} & g_{33}
\ea
\right], \quad h = \left[
\ba{cc}
h_{11} & h_{12} \\
h_{21} & h_{22}
\ea
\right].
\]
Inside $G$ we distinguish two subgroups: a reductive subgroup $G_{\text{red}}$
given by the conditions $u_{21}=0$, $u_{31}=0$ and a unitary subgroup $G'$
consisting of pairs $(g,h)$ of the form
\[
g = \left[
\ba{ccc}
1 & 0 & 0 \\
u_{21} & 1 & 0 \\
u_{31} & 0 & 1
\ea
\right], \quad h = \left[
\ba{cc}
1 & 0 \\
0 & 1
\ea
\right].
\]
Consider the open $G$-invariant subset $U \subset \U$ of morphisms
\[
\f = \left[
\ba{ccc}
q_1 & \ell_{11} & \ell_{12} \\
q_2 & \ell_{21} & \ell_{22}
\ea
\right] \quad \text{for which} \quad \left|
\ba{cc}
\ell_{11} & \ell_{12} \\
\ell_{21} & \ell_{22}
\ea
\right| \neq 0 \quad \text{and} \quad \left|
\ba{cc}
q_1 & \ell_{11} \\
q_2 & \ell_{21}
\ea
\right|, \quad \left|
\ba{cc}
q_1 & \ell_{12} \\
q_2 & \ell_{22}
\ea
\right|
\]
have linearly independent images in $\SS^3V^*/(\ell_{11} \ell_{22} - \ell_{12} \ell_{21})V^*$.

\begin{prop}
\label{4.2}
There exists a geometric quotient $U/G$ and it is a smooth projective
variety of dimension $10$. There exists a geometric quotient $U/G'$ and the canonical morphism
$U/G' \to U/G$ is a geometric quotient for the induced action of $G_{\text{red}}$ on $U/G'$.
\end{prop}

\begin{proof}
Let $\f$ be in $\U$. As mentioned at \ref{4.1}, $\f$ belongs to $U$ precisely if it is not
in the orbit of a morphism represented by a matrix having one of the following forms:
\[
\f_1 = \left[
\ba{ccc}
\star & 0 & 0 \\
\star & \star & \star
\ea
\right], \quad \f_2 = \left[
\ba{ccc}
\star & \star & 0 \\
\star & \star & 0
\ea
\right], \quad \f_3 = \left[
\ba{ccc}
0 & 0 & \star \\
\star & \star & \star
\ea
\right], \quad \f_4 = \left[
\ba{ccc}
0 & \star & \star \\
0 & \star & \star
\ea
\right].
\]
We will now use the theory of quotients modulo non-reductive groups developed in \cite{drezet-trautmann}.
Let $\L = (\l_1, \l_2, \m_1)$ be a polarisation for the action of $G$ on $\U$ satisfying the
conditions:
\[
\m_1 + 2\l_2 >1, \quad 2\m_1 + \l_2 > 1, \quad \m_1 + \l_1 + \l_2 > 1, \quad 2\m_1 + \l_1 > 1,
\]
\[
\m_1 + \l_1 < 1, \qquad \qquad \m_1 + \l_2 < 1.
\]
We claim that the set of semi-stable points relative to $\L$ coincides with the set of stable
points and coincides with $U$. This follows from King's criterion of semi-stability \cite{king}.
In our situation this criterion says the following (put $\l_3=\l_2$, $\m_2 =\m_1$):
a morphism $\f \in \U$ is semi-stable relative to $\L$ (respectively stable) precisely if for any morphism
$\psi$ in the orbit of $\f$, for any zero-submatrix $A$ of the matrix representing $\psi$,
the sum of the $\m_i$, $i=1, 2$, corresponding to the rows of $A$ and the $\l_j$, $j=1,2,3$, corresponding to the
columns of $A$ is at most $1$ (respectively less than $1$ in the stable case).
Take $\f \in \U^{ss}(\L)$. The first four conditions on $\L$ imply that $\f \nsim \f_i$
for $1 \le i \le 4$, hence $\f$ belongs to $U$. Given $\f$ in $U$, the only possible morphisms
$\psi$ in the orbit of $\f$ that are represented by a matrix having a zero-submatrix
must have one of the following forms:
\[
\psi_1 = \left[
\ba{ccc}
0 & \star & \star \\
\star & \star & \star
\ea
\right], \qquad \psi_2 = \left[
\ba{ccc}
\star & 0 & \star \\
\star & \star & \star
\ea
\right].
\]
In order to ensure that $\f$ be stable we must check the condition arising from King's criterion
for each possible equivalence $\f \sim \psi_i$.
For the equivalence $\f \sim \psi_1$ the condition $\m_1 + \l_1 < 1$ arises.
For the equivalence $\f \sim \psi_2$ the condition $\m_1 + \l_2 < 1$ arises.
These conditions are fulfilled by hypothesis. Thus $\f$ belongs to $\U^s(\L)$.
We have shown the inclusions $\U^{ss}(\L) \subset U \subset \U^s(\L)$,
hence these three sets coincide.

Relations 3.3.1 \cite{drezet-trautmann} for our situation read as follows:
$\l_1 > 0$, $\l_2 > 0$, $\l_1 + 2\l_2 =1$, $\m_1 = 1/2$.
In view of these relations, the polarisations $\L$ satisfying the six conditions from above
are precisely the polarisations $\L = (\l_1, \l_2, 1/2)$ for which $1/4 < \l_2 < 1/2$.

We quote below conditions 7.2.2 and 8.1.3 from op.cit. applicable to our
situation:
\[
\a_1 > 0, \qquad \a_2 > 0, \qquad \l_2 \ge \frac{a_{21}}{2} c_1(2), \qquad \l_2 \ge c_1(2) a_{21} \m_1.
\]
According to op.cit., propositions 6.1.1, 7.2.2 and 8.1.3, the above conditions are sufficient to ensure the
existence of a projective good quotient $\U^{ss}/\!\!/G$ which contains a smooth geometric quotient
$\U^s/G$ as an open subset.
Here $\a_1 = \l_1$, $\a_2 = \l_2 - a_{21}\l_1$, $a_{21} = \dim(\Hom(\O(-3),\O(-2)))=3$.
The constant $c_1(2)$ is defined at 7.1 op.cit.
According to remark 9.4.1 in the preliminary version of \cite{drezet-trautmann},
we have $c_1(2)=1/5$. The four sufficient conditions from above are equivalent to the inequalities
$3/7 < \l_2 < 1/2$.
Fix now a polarisation $\L$ satisfying these conditions.
We have $U = \U^{ss}(\L)= \U^s(\L)$, hence a geometric quotient $U/G$ exists and is a smooth
projective variety.

\medskip

\noi
Next we prove the existence of a geometric quotient $U/G'$.
Consider the open subset $U_0$ of injective morphisms inside $\Hom(2\O(-2), 2\O(-1))$.
Let $U_1$ be the trivial vector bundle over $U_0$ with fibre $\Hom(\O(-3), 2\O(-1))$.
$U$ is an open $G$-invariant subset of $U_1$.
Let $S \subset U_1$ be the sub-bundle with fibre $\{ \a u \mid u \in \Hom(\O(-3), 2\O(-2)) \}$
at every point $\a \in U_0$.
The quotient bundle $Q = U_1/S$ is a geometric quotient of $U_1$ modulo the action of $G'$.
The image of $U$ in $U_1/G'$ is the geometric quotient $U/G'$. Let $\pi' \colon U \to U/G'$
denote the quotient map.

The quotient map $\pi \colon U \to U/G$ is $G'$-equivariant, hence it factors through
a surjective morphism $\r \colon U/G' \to U/G$. We consider the action of $G_{\text{red}}$ on $U/G'$
defined by $g \pi'(\f) = \pi'(g \f)$. Clearly this action is well-defined and the fibres of $\r$
are $G_{\text{red}}$-orbits.

Let $F \subset U/G'$ be a closed $G_{\text{red}}$-invariant subset.
$\pi'^{-1}(F)$ is closed and $G$-invariant, hence $\r(F) = \pi(\pi'^{-1}(F))$ is closed in $U/G$.

Let $D \subset U/G$ be an open subset and let $f \colon D \to \C$ be a function such that $f \circ \r$
is regular. The composition $f \circ \pi = f \circ \r \circ \pi'$ is regular hence, by the fact that $\pi$ is a geometric
quotient map, we deduce that $f$ is regular. Thus $\r^*$ maps $\O(D)$ isomorphically to
$\O(\r^{-1}(D))^{G_{\text{red}}}$.

Let $D \subset U/G$ be an open affine subset. Since $\pi$ is affine, $\pi^{-1}(D)$
is an open affine subset of $U$.
Notice that $\pi' \colon U \to U/G'$ is a locally trivial principal $G'$-bundle with fibre $G'$.
It follows that $\r^{-1}(D)$ can be identified with a closed
subvariety of $\pi'^{-1}(\r^{-1}(D))= \pi^{-1}(D)$, hence $\r^{-1}(D)$ is affine.
Thus $\r$ is an affine map. From all that was said above we conclude that $\r$
is a geometric quotient map.
\end{proof}

\noi
Let $\W_2 = \Hom(\O(-3) \oplus 2\O(-2) \oplus \O(-1), 2\O(-1) \oplus 2\O)$ and let $W_2 \subset \W_2$
be the set of morphisms $\f$ from proposition \ref{4.1}.
Let
\[
G_2 = (\Aut(\O(-3) \oplus 2\O(-2) \oplus \O(-1)) \times \Aut(2\O(-1) \oplus 2\O))/\C^*
\]
be the natural group acting by conjugation on $\W_2$.
Let $X_2 \subset \M(6,1)$ be the set of isomorphism classes of sheaves of the form $\Coker(\f)$,
$\f \in W_2$. Note that $X_2$ is a locally closed subset; we equip it with the canonical induced reduced
structure.

\begin{prop}
\label{4.3}
There exists a geometric quotient $W_2/G_2$ and it is a smooth quasi-projective
variety of dimension $33$. Let $Y$ be the geometric quotient $U/G$ from proposition \ref{4.2}.
$W_2/G_2$ is an open subset inside a fibre bundle with fibre $\P^{21}$ and base $Y \times \P^2$.
\end{prop}

\begin{proof}
The proof is almost identical to the proof of 3.2.3 \cite{mult_five} with notational differences only.
Let $W_2' \subset \W_2$ be the locally closed subset of morphisms $\f$ satisfying the conditions
from proposition \ref{4.1} except injectivity. The pairs of morphisms $(\f_{11}, \f_{22})$ form an open subset
$U_1 \subset \Hom(\O(-3) \oplus 2\O(-2), 2\O(-1))$ and the morphisms $\f_{23}$ form an open subset
$U_2 \subset \Hom(\O(-1), 2\O)$. We denote $U = U_1 \times U_2$.
$W_2'$ is the trivial vector bundle on $U$ with fibre $\Hom(\O(-3) \oplus 2\O(-2), 2\O)$.
We represent the elements of $G_2$ by pairs $(g,h)$ of matrices
\[
g = \left[
\ba{cccc}
g_{11} & 0 & 0 & 0 \\
u_{21} & g_{22} & g_{23} & 0 \\
u_{31} & g_{32} & g_{33} & 0 \\
u_{41} & u_{42} & u_{43} & g_{44}
\ea
\right], \quad h= \left[
\ba{cccc}
h_{11} & h_{12} & 0 & 0 \\
h_{21} & h_{22} & 0 & 0 \\
v_{31} & v_{32} & h_{33} & h_{34} \\
v_{41} & v_{42} & h_{43} & h_{44}
\ea
\right].
\]
Inside $G_2$ we distinguish four subgroups: a reductive subgroup ${G_2}_{\text{red}}$
given by the conditions $u_{ij}=0$, $v_{ij}=0$, the subgroup $S$ of pairs $(g, h)$ of the form
\[
g = \left[
\ba{cccc}
a & 0 & 0 & 0 \\
0 & a & 0 & 0 \\
0 & 0 & a & 0 \\
0 & 0 & 0 & b
\ea
\right], \quad h = \left[
\ba{cccc}
a & 0 & 0 & 0 \\
0 & a & 0 & 0 \\
0 & 0 & b & 0 \\
0 & 0 & 0 & b
\ea
\right],
\]
with $a, b \in \C^*$, and two unitary subgroups $G_2'$, $G_2''$.
Here $G_2'$ consists of pairs $(g, h)$ of morphisms of the form
\[
g = \left[
\ba{cccc}
1 & 0 & 0 & 0 \\
0 & 1 & 0 & 0 \\
0 & 0 & 1 & 0 \\
u_{41} & u_{42} & u_{43} & 1
\ea
\right], \quad h = \left[
\ba{cccc}
1 & 0 & 0 & 0 \\
0 & 1 & 0 & 0 \\
v_{31} & v_{32} & 1 & 0 \\
v_{41} & v_{42} & 0 & 1
\ea
\right],
\]
while $G_2''$ is given by pairs $(g, h)$, where
\[
g = \left[
\ba{cccc}
1 & 0 & 0 & 0 \\
u_{21} & 1 & 0 & 0 \\
u_{31} & 0 & 1 & 0 \\
0 & 0 & 0 & 1
\ea
\right], \quad h = \left[
\ba{cccc}
1 & 0 & 0 & 0 \\
0 & 1 & 0 & 0 \\
0 & 0 & 1 & 0 \\
0 & 0 & 0 & 1
\ea
\right].
\]
Note that $G_2 = G_2' G_2'' {G_2}_{\text{red}}$. Consider the $G_2$-invariant subset
$\Si \subset W_2'$ given by the conditions
\[
\f_{21} = v \f_{11} + \f_{23} w_1, \quad \f_{22} = v \f_{12} + \f_{23} w_2,
\]
\[
v \in \Hom(2\O(-1), 2\O), \quad w_1 \in \Hom(\O(-3), \O(-1)), \quad w_2 \in \Hom(2\O(-2), \O(-1)).
\]
$W_2$ is the subset of injective morphisms inside $W_2' \setminus \Si$.
We will construct a geometric quotient of $W_2' \setminus \Si$ modulo $G_2$ and it will follow
that $W_2/G_2$ exists and is an open subset of $(W_2' \setminus \Si)/G_2$.

It is easy to see that $\Si$ is a sub-bundle of $W_2'$. The quotient bundle, denoted $E'$, has rank $22$.
The quotient map $W_2' \to E'$ is a geometric quotient modulo $G_2'$.
The canonical action of $G_2'' {G_2}_{\text{red}}$ on $U$ is $E'$-linearised
and the map $W_2' \to E'$ is $G_2'' {G_2}_{\text{red}}$-equivariant.

Next we construct a geometric quotient of $E'$ modulo $G_2''$.
The quotient for the base $U$ is $(U_1/G') \times U_2$, where $G'$ is the group from proposition \ref{4.2}.
The quotient $U_1/G'$ was constructed in the proof of loc.cit., where it was noticed that the quotient
map $U_1 \to U_1/G'$ is a locally trivial principal $G'$-bundle with fibre $G'$.
Thus $U \to U/G_2''$ is a locally trivial principal $G_2''$-bundle with fibre $G_2''$.
According to 4.2.14 \cite{huybrechts}, $E'$ descends to a vector bundle $E$ over $U/G_2''$.
The canonical map $E' \to E$ is a geometric quotient modulo $G_2''$.
The composed map $W_2' \to E' \to E$ is a geometric quotient modulo $G_2' G_2''$.
The canonical induced action of ${G_2}_{\text{red}}$ on $U/G_2''$ is linearised with respect to $E$
and the map $W_2' \to E$ is ${G_2}_{\text{red}}$-equivariant.

Let $x \in U/G_2''$ be a point and let $\xi \in E_x$ be a non-zero vector lying over $x$.
The stabiliser of $x$ in ${G_2}_{\text{red}}$ is $S$ and $S \xi = \C^* \xi$.
Let $\s$ be the zero-section of $E$. The canonical map $E \setminus \s \to \P(E)$
is a geometric quotient modulo $S$.

We finally construct a geometric quotient of $\P(E)$ modulo the induced action of ${G_2}_{\text{red}}/S$.
By proposition \ref{4.2}, $(U/G_2'')/({G_2}_{\text{red}}/S)$ exists and is isomorphic to the smooth $12$-dimensional
projective variety $Y \times \P^2$.
We consider the character $\chi$ of ${G_2}_{\text{red}}$ given by $\chi(g,h)= \det(g) \det(h)^{-1}$.
We multiply the action of ${G_2}_{\text{red}}$ on $E$ by $\chi$ and we denote the resulting
linearised bundle by $E_{\chi}$.
The action of $S$ on $E_{\chi}$ is trivial, hence $E_{\chi}$ is ${G_2}_{\text{red}}/S$-linearised.
The isotropy subgroup in ${G_2}_{\text{red}}/S$ for any point in $U/G_2''$ is trivial,
so we can apply \cite{huybrechts}, lemma 4.2.15, to deduce that $E_{\chi}$ descends to a vector bundle
$F$ over $Y \times \P^2$. The map $E_{\chi} \to F$ is a geometric quotient modulo ${G_2}_{\text{red}}/S$.
The same can be said of the map $\P(E_{\chi}) \to \P(F)$.

We conclude by observing that the composed map $W_2' \setminus \Si \to E \setminus \s \to \P(E) \to \P(F)$
is a geometric quotient modulo $G_2$ and that $W_2/G_2$ exists and is a proper open subset of $\P(F)$.
\end{proof}

\begin{prop}
\label{4.4}
The geometric quotient $W_2/G_2$ is isomorphic to $X_2$.
\end{prop}

\begin{proof}
As at \ref{3.2}, we will show that the canonical morphism $\r \colon W_2 \to X_2$
is a categorical quotient map.
The isomorphism $W_2/G_2 \isom X_2$ will follow from the uniqueness of the categorical quotient.
Consider the sheaf $\G = \F^\D(1)$.
Diagram (2.2.3) \cite{drezet-maican} for $\G$ takes the form
\[
\xymatrix
{
2\O(-2) \ar[r]^-{\f_1} & \O(-1) & 0 \\
\O(-2) \ar[r]^-{\f_3} & 5\O(-1) \ar[r]^-{\f_4} & 5\O
}.
\]
$\Coker(\f_1)$ is isomorphic to the structure sheaf $\C_x$ of a point because it is a quotient sheaf of $\G$.
Thus $\Ker(\f_1) \isom \O(-3)$ and the exact sequence (2.2.5) \cite{drezet-maican} reads
\[
0 \lra \O(-3) \stackrel{\f_5}{\lra} \Coker(\f_4) \lra \G \lra \C_x \lra 0.
\]
We see from this that $\Coker(\f_4)$ has no zero-dimensional torsion, which allows us to deduce,
as in the proof of 3.2.4 \cite{mult_five}, that $\Coker(\f_3) \isom 2\O(-1) \oplus \Om^1(1)$.
From (2.2.4) \cite{drezet-maican} we get the exact sequence
\[
0 \lra 2\O(-1) \oplus \Om^1(1) \lra 5\O \lra \Coker(\f_4) \lra 0.
\]
We now apply the horseshoe lemma to the extension
\[
0 \lra \Coker(\f_5) \lra \G \lra \C_x \lra 0
\]
and to the resolutions
\[
0 \lra \O(-3) \lra \Coker(\f_4) \lra \Coker(\f_5) \lra 0,
\]
\[
0 \lra \O(-3) \lra 2\O(-2) \lra \O(-1) \lra \C_x \lra 0.
\]
The morphism $\O(-1) \to \C_x$ lifts to a morphism $\O(-1) \to \G$ because
$\H^1(\Coker(\f_5)(1))$ vanishes. We obtain the resolution
\[
0 \lra \O(-3) \lra \O(-3) \oplus 2\O(-2) \lra \O(-1) \oplus \Coker(\f_4) \lra \G \lra 0.
\]
The map $\O(-3) \to \O(-3)$ above is non-zero because $\H^1(\G)$ vanishes.
Canceling $\O(-3)$ and combining with the above resolution of $\Coker(\f_4)$
yields the exact sequence
\[
0 \lra 2\O(-2) \oplus 2\O(-1) \oplus \Om^1(1) \lra \O(-1) \oplus 5\O \lra \G \lra 0.
\]
Dualising we get the exact sequence
\[
0 \lra 5\O(-2) \oplus \O(-1) \lra \Om^1 \oplus 2\O(-1) \oplus 2\O \lra \F \lra 0.
\]
As in the proof of \ref{3.1}, the above leads to a resolution of the form
\[
0 \lra \O(-3) \oplus 2\O(-2) \oplus \O(-1) \lra 2\O(-1) \oplus 2\O \lra \F \lra 0
\]
in which the map $\O(-1) \to 2\O(-1)$ is zero. In conclusion, we have obtained resolution \ref{4.1}
in a natural manner from the Beilinson spectral sequence of $\F$ (or of $\G$).
This allows us to conclude, as at 3.1.6 \cite{drezet-maican}, that $\r$ is a categorical quotient map.
\end{proof}

%%%%%%%%%%%%%%%%%%%%%%%%%%%%%%%%%%%%% section 5

\section{The codimension $6$ stratum}

\begin{prop}
\label{5.1}
The sheaves $\F$ in $\M(6,1)$ satisfying the conditions $\h^1(\F)=2$ and
$\h^0(\F(-1))=0$ are precisely the sheaves with resolution
\[
0 \lra 2\O(-3) \oplus 2\O(-1) \stackrel{\f}{\lra} \O(-2) \oplus 3\O \lra \F \lra 0,
\]
where $\f_{11}$ has linearly independent entries and $\f_{22}$ has linearly independent maximal minors.
\end{prop}

\begin{proof}
Assume that $\F$ gives a point in $\M(6,1)$ and satisfies the cohomological
conditions from the proposition. The sheaf $\G=\F^{\D}(1)$ gives a point in $\M(6,5)$
and satisfies the dual conditions $\h^0(\G(-1))=2$ and $\h^1(\G)=0$.
We put $m= \h^1(\G \tensor \Om^1(1))$.
The Beilinson free monad yields the resolution
\[
0 \lra 2\O(-2) \stackrel{\eta}{\lra} 3\O(-2) \oplus (m+4)\O(-1) \stackrel{\f}{\lra} m\O(-1) \oplus 5\O \lra \G \lra 0.
\]
\[
\eta = \left[
\ba{c}
0 \\ \psi
\ea
\right].
\]
Here $\f_{12} =0$. As $\G$ has rank zero and maps surjectively onto $\CC=\Coker(\f_{11})$,
we have $m \le 3$. If $m=3$, then $\CC$ has Hilbert polynomial $\PP(t)=3t$, so it is a destabilising
quotient sheaf of $\G$. The cases $m=0$ and $m=1$ can be eliminated using the arguments
from the proof of 3.1.3 \cite{mult_five}. Thus $m=2$.
As in the proof of 3.2.5 \cite{mult_five}, we may assume that $\psi$ is represented by the matrix
\[
\left[
\ba{cccccc}
X & Y & Z & 0 & 0 & 0 \\
0 & 0 & 0 & X & Y & Z
\ea
\right]^\T.
\]
By duality, we obtain a monad for $\F$ of the form
\[
0 \lra 5\O(-2) \oplus 2\O(-1) \lra 6\O(-1) \oplus 3\O \stackrel{\eta^\T}{\lra} 2\O \lra 0,
\]
yielding a resolution
\[
0 \lra 5\O(-2) \oplus 2\O(-1) \lra 2 \Om^1 \oplus 3\O \lra \F \lra 0.
\]
From this we get a resolution
\[
0 \lra 2\O(-3) \oplus 5\O(-2) \oplus 2\O(-1) \stackrel{\f}{\lra} 6\O(-2) \oplus 3\O \lra \F \lra 0
\]
in which $\rank(\f_{12})=5$. We finally arrive at the resolution of $\F$ from the proposition.
The conditions on $\f$ in the proposition follow from the semi-stability of $\F$.
If $\f_{11}$ had linearly dependent entries, then $\F$ would map surjectively onto
a sheaf with Hilbert polynomial $\PP(t)=t-1$. If $\f_{22}$ had linearly dependent maximal minors,
then it would be equivalent to a morphism
represented by a matrix with a zero-row or a zero-submatrix of size $2 \times 1$.
Thus $\F$ would have a destabilising subsheaf with Hilbert polynomial $\PP(t)=2t+2$ or $t+1$.

\medskip

\noi
Conversely, we assume that $\F$ has a resolution as in the proposition and we need to show that
there are no destabilising subsheaves. From the snake lemma we get an extension
\[
0 \lra \F' \lra \F \lra \C_x \lra 0,
\]
in which $\C_x$ is the structure sheaf of the point given by the ideal generated by the entries
of $\f_{11}$ and $\F'$ has resolution of the form
\[
0 \lra \O(-4) \oplus 2\O(-1) \stackrel{\psi}{\lra} 3\O \lra \F' \lra 0
\]
in which $\psi_{12}=\f_{22}$.
According to proposition \ref{5.2} below, $\F'$ is semi-stable and the only possible
subsheaves of $\F'$ of slope zero must be of the form $\O_L(-1)$ for a line $L \subset \P^2$.
It follows that for every subsheaf $\E \subset \F$ we have $\pp(\E) \le 0$ except, possibly,
subsheaves that fit into an extension of the form
\[
0 \lra \O_L(-1) \lra \E \lra \C_x \lra 0.
\]
We have $\E \isom \O_L$ because $\E$ has no zero-dimensional torsion.
If $\F$ had such a subsheaf, then we would get a commutative diagram with exact rows
\[
\xymatrix
{
0 \ar[r] & \O(-1) \ar[r] \ar[d]^-{\b} & \O \ar[r] \ar[d]^-{\a} & \O_L \ar[r] \ar[d] & 0 \\
0 \ar[r] & 2\O(-3) \oplus 2\O(-1) \ar[r]^-{\f} & \O(-2) \oplus 3\O \ar[r] & \F \ar[r] & 0
},
\]
in which $\a$ is injective, because it is injective on global sections, hence $\b$
is also injective. It would follow that $\f$ is equivalent to a morphism represented by a matrix
with only one non-zero entry on the last column. This would violate our hypothesis on $\f$.
\end{proof}

\begin{prop}
\label{5.2}
Let $\F$ be a sheaf with resolution of the form
\[
0 \lra \O(-4) \oplus 2\O(-1) \stackrel{\psi}{\lra} 3\O \lra \F \lra 0,
\]
in which $\psi_{12}$ has linearly independent maximal minors. Then $\F$ is semi-stable,
i.e. it gives a point in $\M(6, 0)$. If the maximal minors of $\psi_{12}$ have no common factor,
then $\F$ is stable. If they have a common linear factor $\ell$, then $\O_L(-1) \subset \F$
is the unique proper subsheaf of slope zero, where $L \subset \P^2$ is the line with equation $\ell =0$.
\end{prop}

\begin{proof}
Assume that the maximal minors of $\psi_{12}$ have no common factor.
By analogy with 2.3.4(i) \cite{mult_five}, $\F$ is isomorphic to $\J_Z(2)$,
where $Z \subset \P^2$ is a zero-dimensional scheme of length $3$ not contained in a line,
contained in a sextic curve $C$, and $\J_Z \subset \O_C$ is its ideal sheaf.
Let $\I$ be a subsheaf of $\J_Z$. According to 6.7 \cite{maican}, there is a sheaf $\J$ such that
$\I \subset \J \subset \O_C$, $\J/\I$ is supported on finitely many points and $\O_C/\J$ is isomorphic
to the structure sheaf of a curve $C'$ contained in $C$.
Excluding the uninteresting case when $\J$ has multiplicity $6$, we have the following possibilities
for the Hilbert polynomial $\PP(t)$ of $\J$, depending on the degree of $C'$:
$5t-10$, $4t-10$, $3t-9$, $2t-7$, $t-4$.
From this we see that $\pp(\I) < \pp(\J_Z)$ except in the case when $\I = \J$ and $\PP_{\J}(t) = 5t-10$,
i.e. when $\J$ is the kernel of a surjective morphism $\O_C \to \O_L$ for a line $L \subset C$.
This situation is ruled out if we take into account that $Z$ may not be a subscheme of a line.

Assume now that the maximal minors of $\psi_{12}$ have a common linear factor $\ell$.
As at 2.3.4(ii) \cite{mult_five}, we have an extension
\[
0 \lra \O_L(-1) \lra \F \lra \O_C(1) \lra 0,
\]
where $L$ is the line with equation $\ell=0$ and $C$ is a quintic curve.
Thus $\F$ is semi-stable and $\O_L(-1)$, $\O_C(1)$ are its stable factors.
The latter cannot be a subsheaf of $\F$ because $\H^0(\F(-1))$ vanishes.
\end{proof}

\noi
Let $\W_3 = \Hom(2\O(-3) \oplus 2\O(-1), \O(-2) \oplus 3\O)$ and let $W_3 \subset \W_3$
be the set of morphisms $\f$ from proposition \ref{5.1}.
Let
\[
G_3= (\Aut(2\O(-3) \oplus 2\O(-1)) \times \Aut(\O(-2) \oplus 3\O))/\C^*
\]
be the natural group acting by conjugation on $\W_3$.
Let $X_3 \subset \M(6,1)$ be the set of isomorphism classes of sheaves of the form
$\Coker(\f)$, $\f \in W_3$.
$X_3$ is a locally closed subset which we equip with the canonical induced reduced structure.

\begin{prop}
\label{5.3}
The generic sheaves in $X_3$ have the form $\O_C(2)(-P_1 -P_2 -P_3 +P_4)$,
where $C \subset \P^2$ is a smooth sextic curve, $P_i$, $1 \le i \le 4$, are distinct points on $C$
and $P_1, P_2, P_3$ are non-colinear.
\end{prop}

\begin{proof}
According to propositions \ref{5.1} and \ref{5.2}, the generic sheaves in $X_3$ are precisely the non-split extension
sheaves of the form
\[
0 \lra \J_Z(2) \lra \F \lra \C_x \lra 0,
\]
where $Z \subset \P^2$ is a zero-dimensional scheme of length $3$ not contained in a line,
contained in a sextic curve $C$, and $\J_Z \subset \O_C$ is its ideal sheaf.
Take $C$ to be smooth, take $Z$ to be the union of three distinct points different from $x$.
Then $\F \isom \O_C(2)(-P_1-P_2-P_3+x)$. Conversely, it can be easily seen that any such
sheaf is in $X_3$.
\end{proof}

\begin{prop}
\label{5.4}
There exists a geometric quotient $W_3/G_3$ and it is a proper open subset
inside a fibre bundle over $\P^2 \times \N(3,2,3)$ with fibre $\P^{23}$.
$W_3/G_3$ is isomorphic to $X_3$.
\end{prop}

\begin{proof}
The proof for the first statement is identical to the proof of 2.2.2 \cite{mult_five}.
Let $W_3'$ be the locally closed subset of $\W_3$ given by the following conditions:
$\f_{12}=0$, $\f_{11}$ has linearly independent entries, $\f_{22}$ has linearly independent
maximal minors. Let $\Si \subset W_3'$ be the $G_3$-invariant subset given by the condition
\[
\f_{21}= \f_{22} u + v \f_{11}, \quad u \in \Hom(2\O(-3), 2\O(-1)), \quad v \in \Hom(\O(-2), 3\O).
\]
As at loc.cit., we can construct a vector bundle $Q$ over $\P^2 \times \N(3,2,3)$ of rank $24$,
such that $\P(Q)$ is a geometric quotient of $W_3' \setminus \Si$ modulo $G_3$.
Then $W_3/G_3$ is a proper open subset of $\P(Q)$.

Let $\F$ give a point in $X_3$ and let $\G= \F^\D(1)$. The Beilinson tableau (2.2.3) \cite{drezet-maican}
for $\G$ takes the form
\[
\xymatrix
{
3\O(-2) \ar[r]^-{\psi_1} & 2\O(-1) & 0 \\
2\O(-2) \ar[r]^-{\psi_3} & 6\O(-1) \ar[r]^-{\psi_4} & 5\O
}.
\]
The exact sequence (2.2.5) \cite{drezet-maican} for this situation reads:
\[
0 \lra \Ker(\psi_1) \stackrel{\psi_5}{\lra} \Coker(\psi_4) \lra \G \lra \Coker(\psi_1) \lra 0.
\]
We see from the above that $\Coker(\psi_4)$ has no zero-dimensional torsion
and that there are no non-zero morphisms $\O_L(1) \to \Coker(\psi_4)$ for any line $L \subset \P^2$.
This allows us to deduce, as at \cite{mult_five}, 3.1.3 and 3.2.5, that $\psi_3$ is equivalent
to the morphism represented by the matrix
\[
\left[
\ba{cccccc}
X & Y & Z & 0 & 0 & 0 \\
0 & 0 & 0 & X & Y & Z
\ea
\right]^\T,
\]
i.e. that $\Coker(\psi_3) \isom 2\Om^1(1)$. The Beilinson tableau for $\F$ has the form
\[
\xymatrix
{
5\O(-2) \ar[r]^-{\f_1} & 6\O(-1) \ar[r]^-{\f_2} & 2\O \\
0 & 2\O(-1) \ar[r]^-{\f_4} & 3\O
}.
\]
The morphism $\f_2$ is dual to $\psi_3$, hence $\Ker(\f_2) \isom 2\Om^1$.
Write $\CC= \Ker(\f_2)/\Im(\f_1)$. The exact sequence
\[
0 \lra \Ker(\f_1) \lra 5\O(-2) \lra 2\Om^1 \lra \CC \lra 0
\]
yields the resolution
\[
0 \lra \Ker(\f_1) \lra 2\O(-3) \oplus 5\O(-2) \stackrel{\a}{\lra} 6\O(-2) \lra \CC \lra 0.
\]
$\CC$ has rank zero because it is a quotient sheaf of $\F$.
Thus $\rank(\a_{12}) \ge 4$. If $\rank(\a_{12})=4$, then $\CC$ would map surjectively
onto the cokernel $\CC'$ of a morphism $2\O(-3) \to 2\O(-2)$.
$\CC'$ would then be a destabilising quotient sheaf of $\F$.
Thus $\rank(\a_{12})=5$ and we arrive at an exact sequence
\[
0 \lra \Ker(\f_1) \lra 2\O(-3) \lra \O(-2) \lra \CC \lra 0.
\]
Using the semi-stability hypothesis on $\F$ it is easy to see that $\CC$ is isomorphic
to the structure sheaf of a point and that $\Ker(\f_1)$ is isomorphic to $\O(-4)$.
The exact sequence (2.2.5) \cite{drezet-maican} takes the form
\[
0 \lra \Ker(\f_1) \stackrel{\f_5}{\lra} \Coker(\f_4) \lra \F \lra \CC \lra 0.
\]
We apply the horseshoe lemma to the extension
\[
0 \lra \Coker(\f_5) \lra \F \lra \CC \lra 0,
\]
to the above resolution of $\CC$ and to the resolution
\[
0 \lra \O(-4) \oplus 2\O(-1) \lra 3\O \lra \Coker(\f_5) \lra 0.
\]
We obtain the exact sequence
\[
0 \lra \O(-4) \lra \O(-4) \oplus 2\O(-3) \oplus 2\O(-1) \lra \O(-2) \oplus 3\O \lra \F \lra 0.
\]
The map $\O(-4) \to \O(-4)$ is non-zero because $\h^1(\F)=2$.
Canceling $\O(-4)$ yields resolution \ref{5.1}.
Thus, for a sheaf $\F$ giving a point in $X_3$ we have obtained resolution \ref{5.1}
in a natural manner from the Beilinson spectral sequence.
As at 3.1.6 \cite{drezet-maican}, we conclude that the canonical bijective morphism
$W_3/G_3 \to X_3$ is an isomorphism.
\end{proof}

\begin{prop}
\label{5.5}
$X_3$ lies in the closure of $X_2$.
\end{prop}

\begin{proof}
The argument can be found at 2.1.6 \cite{mult_five}, or at 3.2.3 \cite{drezet-maican}.
Using the Beilinson monad for $\F(-1)$ we see that the open subset $U \subset \M(6,1)$
given by the conditions $\h^0(\F(-1))=0$ and $\h^1(\F(1))=0$
is parametrised by an open subset $M$ inside the space of monads of the form
\[
0 \lra 11\O(-1) \stackrel{A}{\lra} 16\O \stackrel{B}{\lra} 5\O(1) \lra 0.
\]
Consider the map $\Phi \colon M \to \Hom(16\O, 5\O(1))$ defined by $\Phi(A,B)=B$.
Using the vanishing of $\H^1(\F(1))$ for an arbitrary sheaf $\F$ giving a point in $U$,
we can prove that $M$ is smooth and that $\Phi$ has surjective differential at every point.
This further leads to the conclusion that the set of monads in $M$ whose cohomology sheaf $\F$
satisfies the relation $\h^1(\F)=2$ is included in the closure of the set of monads for which $\h^1(\F)=1$.
Thus $X_3$ lies in the relative closure of $X_2$ in $U$, hence $X_3 \subset \overline{X}_2$.
\end{proof}

\begin{prop}
\label{5.6}
The sheaves $\F$ giving a point in $\M(6,1)$ and satisfying the cohomological
conditions $\h^0(\F(-1)) =1$ and $\h^1(\F) = 2$ are precisely the sheaves having a resolution of the form
\[
\tag{i}
0 \lra 2\O(-3) \stackrel{\f}{\lra} \O(-1) \oplus \O(1) \lra \F \lra 0,
\]
\[
\f = \left[
\ba{cc}
q_1 & q_2 \\
g_1 & g_2
\ea
\right],
\]
where $q_1, q_2$ are linearly independent two-forms without a common linear factor, or of the form
\[
\tag{ii}
0 \lra 2\O(-3) \oplus \O(-2) \stackrel{\f}{\lra} \O(-2) \oplus \O(-1) \oplus \O(1) \lra \F \lra 0,
\]
\[
\f= \left[
\ba{ccc}
\ell_1 & \ell_2 & 0 \\
q_1 & q_2 & \ell \\
g_1 & g_2 & h
\ea
\right],
\]
where $\ell_1, \ell_2$ are linearly independent one-forms, $\ell \neq 0$,
and there are no linear forms $u, v_1, v_2$ such that
$(q_1, q_2)= u (\ell_1, \ell_2) + \ell (v_1, v_2)$.
\end{prop}

\begin{proof}
Consider a sheaf $\F$ in $\M(6,1)$ satisfying the cohomological conditions from the proposition.
Put $m=\h^0(\F \tensor \Om^1(1))$. The Beilinson free monad for $\G = \F^\D(1)$ reads
\[
0 \lra 2\O(-2) \lra 3\O(-2) \oplus (m+4) \O(-1) \lra m\O(-1) \oplus 6\O \lra \O \lra 0
\]
and yields the resolution
\[
0 \lra 2\O(-2) \lra 3\O(-2) \oplus (m+4)\O(-1) \lra \Om^1 \oplus (m-3) \O(-1) \oplus 6\O \lra \G \lra 0.
\]
We have $m \ge 3$. Moreover, $m \le 4$
because $\G$ maps surjectively onto the cokernel $\CC$ of the morphism
$3\O(-2) \to \Om^1 \oplus (m-3)\O(-1)$.
If $m=4$, then $\CC$ has Hilbert polynomial $\PP(t)=2t-1$, hence it is a destabilising quotient
sheaf of $\G$. We deduce that $m=3$.
As in the proof of 3.2.5 \cite{mult_five}, we can show that the morphism $2\O(-2) \to 7\O(-1)$
occurring in the resolution is equivalent to the morphism represented by the matrix
\[
\left[
\ba{ccccccc}
X & Y & Z & 0 & 0 & 0 & 0 \\
0 & 0 & 0 & X & Y & Z & 0
\ea
\right]^\T.
\]
Recall that the argument rests on the fact that the only morphism $\O_L(1) \to \G$,
for a line $L \subset \P^2$, is the zero-morphism. Indeed, both $\O_L(1)$ and $\G$ are
semi-stable and the slope of the first sheaf exceeds the slope of the second sheaf.
Thus far we have a resolution
\[
0 \lra 3\O(-2) \oplus 2\Om^1(1) \oplus \O(-1) \lra \Om^1 \oplus 6\O \lra \G \lra 0.
\]
Resolving $\Om^1$ yields the resolution
\[
0 \lra \O(-3) \oplus 3\O(-2) \oplus 2\Om^1(1) \oplus \O(-1) \lra 3\O(-2) \oplus 6\O \lra \G \lra 0.
\]
As in the proof of 2.1.4 \cite{mult_five}, it can be shown that the map $3\O(-2) \to 3\O(-2)$
has rank $3$. We arrive at the resolution
\[
0 \lra \O(-3) \oplus 2\Om^1(1) \oplus \O(-1) \lra 6\O \lra \G \lra 0.
\]
According to \cite{maican-duality}, lemma 3, taking duals yields the resolution
\[
0 \lra 6\O(-2) \lra \O(-1) \oplus 2\Om^1 \oplus \O(1) \lra \F \lra 0.
\]
Resolving $2\Om^1$ leads to the resolution
\[
0 \lra 2\O(-3) \oplus 6\O(-2) \stackrel{\r}{\lra} 6\O(-2) \oplus \O(-1) \oplus \O(1) \lra \F \lra 0.
\]
As $\r$ is injective, $\rank(\r_{12})$ is at lest $4$. If $\rank(\r_{12})=4$, then $\F$ would have a destabilising
quotient sheaf arising as the cokernel of an injective morphism $2\O(-3) \to 2\O(-2)$.

When $\rank(\r_{12})=6$ we get resolution (i). The conditions on $q_1$ and $q_2$ in the statement
follow from the semi-stability of $\F$.
Thus, if $q_1, q_2$ had a common linear factor, say $q_1 = \ell \ell_1, q_2 = \ell \ell_2$, then we would
get a commutative diagram
\[
\xymatrix
{
& 2\O(-3) \egal[r] \ar[d]^-{\psi} & 2\O(-3) \ar[d]^-{\f} \\
 0 \ar[r] & \O(-2) \oplus \O(1) \ar[r]^-{\l} & \O(-1) \oplus \O(1) \ar[r] & \O_L(-1) \ar[r] & 0
 },
 \]
 \[
 \psi = \left[
 \ba{cc}
 \ell_1 & \ell_2 \\
 g_1 & g_2
 \ea
 \right], \qquad \l = \left[
 \ba{cc}
 \ell & 0 \\
 0 & 1
 \ea
 \right].
 \]
 Here $L$ is the line with equation $\ell =0$.
 From the shake lemma we see that $\F$ would map surjectively into $\O_L(-1)$,
 in violation of semi-stability.

Finally, when $\rank(\r_{12})=5$ we get resolution (ii). Again, the conditions on $\f$ follow from the
semi-stability of $\F$. For instance, if $(q_1, q_2)= u (\ell_1, \ell_2) + \ell (v_1, v_2)$,
then $\f$ would be equivalent to the morphism represented by the matrix
\[
\left[
\ba{ccc}
\ell_1 & \ell_2 & 0 \\
0 & 0 & \ell \\
g_1 & g_2 & h
\ea
\right],
\]
hence $\O_L(-1)$ would be a destabilising quotient sheaf of $\F$. Here $L \subset \P^2$
is the line with equation $\ell =0$.

\medskip

\noi
Conversely, we assume that $\F$ has resolution (i) and we need to show that $\F$ is semi-stable.
Equivalently, we need to show that the dual sheaf $\G = \F^\D(1)$ gives a point in $\M(6,5)$.
Taking duals in (i) yields the resolution
\[
0 \lra \O(-3) \oplus \O(-1) \stackrel{\psi}{\lra} 2\O(1) \lra \G \lra 0,
\]
\[
\psi = \left[
\ba{cc}
g_1 & q_1 \\
g_2 & q_2
\ea
\right].
\]
Let $Z$ be the zero-dimensional scheme of length $4$ given by the ideal $(q_1, q_2)$.
Let $C \subset \P^2$ be the sextic curve with equation $q_1 g_2 - q_2 g_1 = 0$.
Let $\J_Z \subset \O_C$ be the ideal sheaf of $Z$ inside $C$.
It is clear from the above resolution that $\G$ is isomorphic to $\J_Z(3)$,
so we must show that $\J_Z$ is semi-stable. Let $\S \subset \J_Z$ be a subsheaf.
According to \cite{maican}, lemma 6.7, there is an ideal sheaf $\A \subset \O_C$ containing
$\S$ such that $\A/\S$ is supported on finitely many points and $\O_C/\A \isom \O_S$
for a curve $S \subset C$ of degree $d$.
We may assume that $1 \le d \le 5$. We have
\begin{align*}
\PP_{\S}(t) & = \PP_{\A}(t) - \h^0(\A/\S) \\
& = \PP_{\O_C}(t) - \PP_{\O_S}(t) - \h^0(\A/\S) \\
& = (6-d)t + \frac{(d+3)(d-6)}{2} - \h^0(\A/\S), \\
\pp(\S) & = - \frac{d+3}{2} - \frac{\h^0(\A/\S)}{6-d}.
\end{align*}
Thus $\pp(\S) < -13/6 = \pp(\J_Z)$ unless $d=1$ and $\A = \S$.
But in this case $S$ is a line and $Z$ is a subscheme of $S$.
From Bez\^out's theorem we see that the equation of $S$ divides both $q_1$ and $q_2$,
which is contrary to our hypothesis.
We conclude that $\G$ is semi-stable.

\medskip

\noi
We now assume that $\F$ has resolution (ii) and we aim at showing that $\F$ is semi-stable.
We shall first examine the case when $\ell$ does not divide $h$.
Let $x$ be the point given by the equations $\ell_1 = 0, \ell_2 = 0$.
Let $Z \subset \P^2$ be the zero-dimensional scheme of length $3$ given by the ideal
$(\ell, h)$ and let $\I_Z \subset \O$ be its ideal sheaf.
Let $C \subset \P^2$ be the sextic curve given by the equation $\text{det}(\f)=0$.
Let $\J_Z \subset \O_C$ be the ideal sheaf of $Z$ in $C$.
We apply the snake lemma to an exact diagram
similar to the diagram in the proof of 3.1.2(ii) \cite{mult_five}
to get the exact sequence
\[
0 \lra \O(-4) \lra \I_Z(2) \lra \F \lra \C_x \lra 0,
\]
which leads to the exact sequence
\[
0 \lra \J_Z(2) \lra \F \lra \C_x \lra 0.
\]
Let $\F' \subset \F$ be a subsheaf. Put $\K = \F' \cap \J_Z(2)$ and let $\CC$ be the image of
$\F'$ in $\C_x$. We shall estimate the slope of $\F'$ by the same method as above.
There is a sheaf $\A \subset \O_C(2)$ containing $\K$ such that $\A/\K$ is supported
on finitely many points and $\O_C(2)/\A \isom \O_S(2)$ for a curve $S \subset C$ of degree $d$.
We may assume that $1 \le d \le 5$ and we have
\[
\pp(\F')= \frac{1-d}{2} + \frac{\h^0(\CC) - \h^0(\A/\K)}{6-d}.
\]
Thus $\pp(\F') < 1/6 = \pp(\F)$ unless $d=1$, $\CC = \C_x$ and $\A = \K$.
In this case $\K \isom \O_Q(1)$, where $Q \subset \P^2$ is a quintic curve.
Thus $\F'$ is a non-split extension of $\C_x$ by $\O_Q(1)$.
According to 3.1.5 \cite{mult_five}, $\F'$ has a resolution of the form
\[
0 \lra 2\O(-3) \stackrel{\psi}{\lra} \O(-2) \oplus \O(1) \lra \F' \lra 0.
\]
This resolution must fit into a commutative diagram of the form
\[
\xymatrix
{
0 \ar[r] & 2\O(-3) \ar[d]^-{\b} \ar[r]^-{\psi} & \O(-2) \oplus \O(1) \ar[d]^-{\a} \ar[r] & \F' \ar[r] \ar[d] & 0 \\
0 \ar[r] & 2\O(-3) \oplus \O(-2) \ar[r]^-{\f} & \O(-2) \oplus \O(-1) \oplus \O(1) \ar[r] & \F \ar[r] & 0
}.
\]
Notice that $\a(-1)$ is injective on global sections, hence $\a_{32} \neq 0$.
Moreover, $\Ker(\a)$ cannot be isomorphic to $\O(-2)$ because it is a subsheaf of $2\O(-3)$.
Thus $\a$ is injective, and so is $\b$.
Modulo elementary operations on rows and columns we have four possibilities:
$\a = \a_1$ or $\a=\a_2$ and $\b = \b_1$ or $\b = \b_2$, where
\[
\a_1 = \left[
\ba{cc}
1 & 0 \\
0 & 0 \\
0 & 1
\ea
\right], \qquad \a_2 = \left[
\ba{cc}
0 & 0 \\
u & 0 \\
0 & 1
\ea
\right], \qquad \b_1 = \left[
\ba{cc}
1 & 0 \\
0 & 1 \\
0 & 0
\ea
\right], \qquad \b_2 = \left[
\ba{cc}
1 & 0 \\
0 & 0 \\
0 & v
\ea
\right].
\]
Here $u$ and $v$ are non-zero one-forms.
If $\a = \a_1$ and $\b = \b_1$, then $\f$ is equivalent to a morphism
represented by a matrix of the form
\[
\left[
\ba{ccc}
\star & \star & 0 \\
0 & 0 & \star \\
\star & \star & \star
\ea
\right]
\]
in violation of our hypothesis. If $\a = \a_1$ and $\b = \b_2$, then $\ell = 0$,
again contradicting the hypothesis. When $\a = \a_2$ we obtain the contradictory conclusion
that $\ell_1$ and $\ell_2$ are linearly dependent.
This shows that $\F$ cannot have a destabilising subsheaf $\F'$.

Lastly, we examine the case when $\ell$ divides $h$.
We may assume that $h = 0$.
Let $L$ be the line given by the equation $\ell = 0$ and let $\E$ be the sheaf given by the exact sequence
\[
0 \lra 2\O(-3) \stackrel{\psi}{\lra} \O(-2) \oplus \O(1) \lra \E \lra 0,
\]
\[
\psi = \left[
\ba{cc}
\ell_1 & \ell_2 \\
g_1 & g_2
\ea
\right].
\]
According to 3.1.5 \cite{mult_five}, $\E$ gives a point of $\M(5,1)$.
Since $\F$ is an extension of $\E$ by $\O_L(-1)$, we have $\pp(\F') \le 0$ for every proper subsheaf
$\F' \subset \F$, unless the induced map $\F' \to \E$ is an isomorphism.
However, we have seen above that $\E$ cannot be isomorphic to a subsheaf of $\F$.
\end{proof}

\noi
Let $\W_4 = \Hom(2\O(-3) \oplus \O(-2), \O(-2) \oplus \O(-1) \oplus \O(1))$,
let $W_4 \subset \W_4$ be the set of injective morphisms with semi-stable cokernel.
We claim that $W_4$ is open.
To see this consider the open subset $W_4' \subset \W_4$ of injective morphisms.
The family of sheaves $\Coker(\f)$, $\f \in W_4'$, is flat over $W_4'$.
Semi-stability is an open condition on flat families, hence $W_4$ is open in $W_4'$.
Let
\[
G_4 = (\Aut(2\O(-3) \oplus \O(-2)) \times \Aut(\O(-2) \oplus \O(-1) \oplus \O(1)))/\C^*
\]
be the natural group acting by conjugation on $\W_4$.
Let $X_4 \subset \M(6,1)$ be the set of isomorphism classes of sheaves of the form
$\Coker(\f)$, $\f \in W_4$. We equip $X_4$ with the canonical induced reduced structure.

\begin{prop}
\label{5.7}
There exists a geometric quotient $W_4/G_4$ and it is isomorphic to $X_4$.
\end{prop}

\begin{proof}
Let $\F$ give a point in $X_4$ and let $\G = \F^\D(1)$.
The Beilinson tableau (2.2.3) \cite{drezet-maican} for $\G$ takes the form
\[
\xymatrix
{
3\O(-2) \ar[r]^-{\f_1} & 3\O(-1) \ar[r]^-{\f_2} & \O \\
2\O(-2) \ar[r]^-{\f_3} & 7\O(-1) \ar[r]^-{\f_4} & 6\O
}.
\]
As at 2.2.4 \cite{mult_five}, we have $\Ker(\f_2) = \Im(\f_1)$ and $\Ker(\f_1) \isom \O(-3)$.
The exact sequence (2.2.5) \cite{drezet-maican} takes the form
\[
0 \lra \O(-3) \stackrel{\f_5}{\lra} \Coker(\f_4) \lra \G \lra 0.
\]
As at 3.2.5 \cite{mult_five}, it can be shown that $\f_3$ is equivalent to the morphism
represented by a matrix of the form
\[
\left[
\ba{ccccccc}
X & Y & Z & 0 & 0 & 0 & 0 \\
0 & 0 & 0 & X & Y & Z & 0
\ea
\right]^\T.
\]
Combining the above resolution of $\G$ with the exact sequence
\[
0 \lra 2\Om^1(1) \oplus \O(-1) \lra 6\O \lra \Coker(\f_4) \lra 0
\]
we obtain the resolution
\[
0 \lra \O(-3) \oplus 2\Om^1(1) \oplus \O(-1) \lra 6\O \lra \G \lra 0.
\]
We have seen at \ref{5.6} how the above leads to a resolution
\[
0 \lra 2\O(-3) \oplus \O(-2) \stackrel{\f}{\lra} \O(-2) \oplus \O(-1) \oplus \O(1) \lra \F \lra 0
\]
with $\f \in W_4$. This construction of $\f$ is natural, so it works for local flat families of
sheaves giving points in $X_4$.
As at 3.1.6 \cite{drezet-maican}, we may infer that the canonical map $W_4 \to X_4$
is a categorical quotient map.
According to \cite{mumford}, remark (2), p. 5, $X_4$ is normal.
Applying \cite{popov-vinberg}, theorem 4.2,
we conclude that $W_4 \to X_4$ is a geometric quotient map.
\end{proof}

\begin{prop}
\label{5.8}
The generic sheaves in $X_4$ are of the form $\O_C(1)(P_1+ P_2+ P_3 + P_4)$,
where $C \subset \P^2$ is a smooth sextic curve and $P_1, P_2, P_3, P_4$ are points on $C$,
no three of which are colinear. In particular, $X_4$ lies in the closure of $X_2$
\end{prop}

\begin{proof}
Let $X_{40} \subset X_4$ be the subset of sheaves from \ref{5.6}(i) and
let $X_{40}^\D \subset \M(6,5)$ be the dual subset.
Dualising the exact sequence \ref{5.6}(i), we see that the sheaves $\G$ giving points in $X_{40}^\D$
are precisely the sheaves having a resolution of the form
\[
0 \lra \O(-3) \oplus \O(-1) \stackrel{\psi}{\lra} 2\O(1) \lra \G \lra 0,
\]
\[
\psi = \left[
\ba{cc}
g_1 & q_1 \\
g_2 & q_2
\ea
\right],
\]
where $q_1, q_2$ have no common factor. Let $Z \subset \P^2$ be the zero-dimensional subscheme
of length $4$ given by the equations $q_1 = 0, q_2 = 0$, let $C \subset \P^2$ be the sextic curve
given by the equation $\det(\psi)=0$, let $\J_Z \subset \O_C$ be the ideal sheaf of $Z$ in $C$.
Clearly, $\G \isom \J_Z(3)$. Conversely, any twisted ideal sheaf $\J_Z(3)$ gives a point in $X_{40}^\D$.
Thus the generic sheaves in $X_4^\D$ have the form $\O_C(3)(-P_1 -P_2 -P_3 -P_4)$, where
$C$ is a somooth sextic curve and $P_1, P_2, P_3, P_4$ are four distinct points on $C$ in general
linear position. The first part of the proposition follows by duality.

To show that $X_4 \subset \overline{X}_2$ we fix a generic sheaf $\G= \O_C(3)(-P_1 -P_2 -P_3 -P_4)$
in $X_4^\D$ as above and we attempt to show that this is in the closure of $X_2^\D$.
We assume, in addition, that the quartic curve with equation $g_2 = 0$ and the conic with equation
$q_2 = 0$ have at least one point of intersection, call it $P_5$, that is distinct from $P_i$, $1 \le i \le 4$.
Choose a sixth point $P_6$ on $C$ distinct from the other five, that converges to $P_5$.
The sheaf $\E = \O_C(3)(-P_1 -P_2 -P_3 -P_4 -P_5 +P_6)$ gives a point in $\M(6,5)$
that converges to the point represented by $\G$.
We claim that $\E$ gives a point in $X_2^\D$.
Assuming the claim to be true, we may conclude that the isomorphism class of $\G$
belongs to $\overline{X}_2^\D$.

It remains to show that $\E$ lies in $X_2^\D$.
Write $q_2 = \ell_{11} \ell_{22} - \ell_{12} \ell_{21}$, $\ell_{ij} \in V^*$, such that $\ell_{12}, \ell_{22}$
are linearly independent and the lines they determine meet at $P_5$.
Consider the sheaf $\E'= \O_C(3)(-P_1 -\ldots -P_5)$ and the zero-dimensional subscheme
$Z' \subset \P^2$ of length $5$ supported on $\{ P_1, \ldots, P_5 \}$.
We have an exact sequence
\[
0 \lra 2\O(-1) \stackrel{\b}{\lra} 2\O \oplus \O(1) \stackrel{\a}{\lra} \O(3) \lra \O_{Z'} \lra 0,
\]
\[
\a = \left[
\ba{ccccc}
-q_1 \ell_{22} & & q_1 \ell_{12} & & \ell_{11} \ell_{22} -\ell_{12} \ell_{21}
\ea
\right],
\]
\[
\b = \left[
\ba{cc}
\ell_{11} & \ell_{12} \\
\ell_{21} & \ell_{22} \\
q_1 & 0
\ea
\right].
\]
Exactness at $2\O \oplus \O(1)$ can be checked directly.
$\Coker(\a)$ has Hilbert polynomial 5 and contains all points of $Z'$ in its support, hence
$\Coker(\a) \isom \O_{Z'}$. Thus $\Im(\a) = \I_{Z'}(3)$, from which we deduce that $\E'$
has resolution
\[
0 \lra \O(-3) \oplus 2\O(-1) \lra 2\O \oplus \O(1) \lra \E' \lra 0.
\]
We apply the horseshoe lemma to the extension
\[
0 \lra \E' \lra \E \lra \C_{P_6} \lra 0,
\]
to the above resolution of $\E'$ and to the standard resolution of $\C_{P_6}$ tensored with $\O(-1)$.
We obtain the exact sequence
\[
0 \lra \O(-3) \lra \O(-3) \oplus 2\O(-2) \oplus 2\O(-1) \lra \O(-1) \oplus 2\O \oplus \O(1) \lra \E \lra 0.
\]
The above extension does not split and $\Ext^1(\C_{P_6}, 2\O \oplus \O(1))=0$,
so we can use the argument at 2.3.2 \cite{mult_five} to conclude that
the morphism $\O(-3) \to \O(-3)$ in the above complex is non-zero.
Cancelling $\O(-3)$ we arrive at the resolution
\[
0 \lra 2\O(-2) \oplus 2\O(-1) \stackrel{\f}{\lra} \O(-1) \oplus 2\O \oplus \O(1) \lra \E \lra 0,
\]
\[
\f = \left[
\ba{cccc}
\ell_1 & \ell_2 & 0 & 0 \\
\star & \star & \ell_{11} & \ell_{12} \\
\star & \star & \ell_{21} & \ell_{22} \\
\star & \star & q_1 & 0
\ea
\right],
\]
where $\ell_1, \ell_2$ are linearly independent one-forms.
It is easy to see that the transpose of $\f$ satisfies the conditions of \ref{4.1}.
In view of our hypothesis on $q_1$ and $q_2$,
the relation $(c_1 \ell_{12} + c_2 \ell_{22})q_1 = u q_2$, $c_1, c_2 \in \C$, $u \in V^*$,
is possible only if $c_1=0$, $c_2 = 0$.
We conclude that $\E$ gives a point in $X_2^\D$.
\end{proof}

%%%%%%%%%%%%%%%%%%%%%%%%%%%%%%%%%%%%%%%%%% section 6

\section{The codimension $8$ stratum}

\begin{prop}
\label{6.1}
The sheaves $\G$ in $\M(6,4)$ satisfying the condition $\h^0(\G(-2)) >0$
are precisely the sheaves with resolution of the form
\[
0 \lra 2\O(-3) \stackrel{\f}{\lra} \O(-2) \oplus \O(2) \lra \G \lra 0,
\]
\[
\f = \left[
\ba{cc}
\ell_1 & \ell_2 \\
f_1 & f_2
\ea
\right],
\]
where $\ell_1, \ell_2$ are linearly independent one-forms.
These sheaves are precisely the non-split extension sheaves of the form
\[
0 \lra \O_C(2) \lra \G \lra \C_x \lra 0,
\]
where $C \subset \P^2$ is a sextic curve and $\C_x$ is the structure sheaf of a point.
\end{prop}

\begin{proof}
The argument is entirely analogous to the argument at 3.1.5 \cite{mult_five}.
\end{proof}

\begin{prop}
\label{6.2}
The sheaves $\G$ in $\M(6,5)$ satisfying the condition $\h^0(\G(-2)) > 0$
are precisely the sheaves with resolution of the form
\[
0 \lra \O(-3) \oplus \O(-2) \stackrel{\f}{\lra} \O(-1) \oplus \O(2) \lra \G \lra 0,
\]
\[
\f = \left[
\ba{cc}
q  & \ell \\
g & h
\ea
\right],
\]
where $\ell \neq 0$ and $\ell$ does not divide $q$. These sheaves are precisely
the extension sheaves of the form
\[
0 \lra \O_C(2) \lra \G \lra \O_Z \lra 0
\]
that do not have zero-dimensional torsion. Here $C \subset \P^2$ is a sextic curve
and $Z \subset \P^2$ is a zero-dimensional scheme of length $2$.
\end{prop}

\begin{proof}
Assume that $\G$ gives a point in $\M(6,5)$ and satisfies the condition $\h^0(\G(-2)) >0$.
As in the proof of 2.1.3 \cite{drezet-maican}, there is an injective morphism $\O_C \to \G(-2)$,
where $C \subset \P^2$ is a curve. Clearly $C$ has degree $6$: if not, $\O_C$ would destabilise $\G(-2)$.
We obtain an extension
\[
0 \lra \O_C(2) \lra \G \lra \CC \lra 0,
\]
where $\CC$ is a sheaf with support of dimension zero and length $2$.
It is clear that $\CC$ is an extension of $\O_{\P^2}$-modules of the form
\[
0 \lra \C_x \lra \CC \lra \C_y \lra 0,
\]
where $\C_x$ and $\C_y$ are the structure sheaves of two points.
Let $\G'$ be the preimage of $\C_x$ in $\G$.
This subsheaf has no zero-dimensional torsion and is an extension of $\C_x$ by $\O_C(2)$
hence, in view of \ref{6.1}, it has a resolution of the form
\[
0 \lra 2\O(-3) \lra \O(-2) \oplus \O(2) \lra \G' \lra 0.
\]
Using the horseshoe lemma, we construct a resolution of $\G$ from the above resolution
of $\G'$ and from the resolution
\[
0 \lra \O(-3) \lra 2\O(-2) \lra \O(-1) \lra \C_y \lra 0.
\]
We obtain a resolution of the form
\[
0 \lra \O(-3) \lra 2\O(-3) \oplus 2\O(-2) \lra \O(-2) \oplus \O(-1) \oplus \O(2) \lra \G \lra 0.
\]
If the morphism $\O(-3) \to 2\O(-3)$ were zero, then it could be shown, as in the proof
of 2.3.2 \cite{mult_five}, that $\C_y$ would be a direct summand of $\G$.
This would contradict our hypothesis. Thus we may cancel $\O(-3)$ to get the resolution
\[
0 \lra \O(-3) \oplus 2\O(-2) \lra \O(-2) \oplus \O(-1) \oplus \O(2) \lra \G \lra 0.
\]
If the morphism $2\O(-2) \to \O(-2)$ were zero, then $\G$ would have a destabilising quotient
sheaf of the form $\O_L(-2)$, for a line $L \subset \P^2$. Thus we may cancel $\O(-2)$
to get the resolution from the proposition.
The conditions on $\ell$ and $q$ follow from the semi-stability of $\G$.

\medskip

\noi
Assume now that $\G$ has a resolution as in the proposition.
$\G$ has no zero-dimensional torsion because it has projective dimension $1$
at every point in its support. Let $Z \subset \P^2$ be the subscheme given by the ideal
$(q, \ell)$ and let $\I_Z \subset \O$ be its ideal sheaf.
Put $f=qh-\ell g$ and let $C$ be the sextic curve with equation $f=0$.
We apply the snake lemma to the commutative diagram with exact rows
\[
\xymatrix
{
0 \ar[r] & \O(-4) \ar[d]^-{f} \ar[r]^-{\scriptsize \left[ \!\! \ba{c} -\ell \\ \phantom{-}q \ea \!\! \right]}
& \O(-3) \oplus \O(-2) \ar[r] \ar[d]^-{\f} & \I_Z(-1) \ar[r] \ar[d] & 0 \\
0 \ar[r] & \O(2) \ar[r]^-{i} & \O(-1) \oplus \O(2) \ar[r]^-{p} & \O(-1) \ar[r] & 0
}.
\]
Here $i$ is the inclusion into the second direct summand and $p$ is the projection
onto the first direct summand.
We deduce that $\G$ is an extension of $\O_Z$ by $\O_C(2)$.

\medskip

\noi
Assume that $\G$ is an extension of $\O_Z$ by $\O_C(2)$ and that it has no zero-dimensional
torsion. Our aim is to show that $\G$ is semi-stable. Let $\G' \subset \G$ be a subsheaf;
denote by $\CC'$ its image in $\O_Z$ and put $\K = \G' \cap \O_C(2)$.
By \cite{maican}, lemma 6.7, there is a twisted ideal sheaf $\A \subset \O_C(2)$
containing $\K$ such that $\A/\K$ is supported on finitely many points
and $\O_C(2)/\A \isom \O_S(2)$ for a curve $S \subset \P^2$ of degree $d$.
We may assume that $1 \le d \le 5$.
We can now estimate the slope of $\G'$ as in the proof of 3.1.2(ii) \cite{mult_five}:
\begin{align*}
\PP_{\G'}(t) & = \PP_{\K}(t) + \h^0(\CC') \\
& = \PP_{\A}(t) - \h^0(\A/\K) + \h^0(\CC') \\
& = \PP_{\O_C}(t+2) - \PP_{\O_S}(t+2) - \h^0(\A/\K) + \h^0(\CC') \\
& = (6-d)t + \frac{(d-1)(d-6)}{2} - \h^0(\A/\K) + \h^0(\CC'), \\
\pp(\G') & = \frac{1-d}{2} + \frac{\h^0(\CC') - \h^0(\A/\K)}{6-d}
\le \frac{1-d}{2} + \frac{2}{6-d} < \frac{5}{6} = \pp(\G).
\end{align*}
We conclude that $\G$ is semi-stable, i.e. it gives a point in $\M(6,5)$.
\end{proof}

\begin{prop}
\label{6.3}
The sheaves $\F$ in $\M(6,1)$ satisfying the condition $\h^1(\F(1)) > 0$
are precisely the sheaves with resolution of the form
\[
0 \lra \O(-4) \oplus \O(-1) \stackrel{\f}{\lra} \O \oplus \O(1) \lra \F \lra 0,
\]
\[
\f= \left[
\ba{cc}
h & \ell \\
g & q
\ea
\right],
\]
where $\ell \neq 0$ and $\ell$ does not divide $q$. These are precisely the twisted
ideal sheaves $\J_Z(2)$, where $Z \subset \P^2$ is a zero-dimensional scheme of length $2$
contained in a sextic curve $C$ and $\J_Z \subset \O_C$ is its ideal sheaf.
\end{prop}

\begin{proof}
The first statement follows from proposition \ref{6.2} by duality.
To prove the second statement we notice that the restriction of $\f$ to $\O(-1)$
has cokernel $\I_Z(2)$, where $Z \subset \P^2$ is the subscheme given by the ideal
$(\ell, q)$ and $\I_Z \subset \O_{\P^2}$ is its ideal sheaf.
Thus $\F$ is the cokernel of the induced injective morphism $\O(-4) \to \I_Z(2)$.
The sextic curve defined by the inclusion $\O(-4) \subset \I_Z(2) \subset \O(2)$
has equation $hq - \ell g =0$ and it is clear that $\F \isom \J_Z(2)$.
\end{proof}

\noi
Let $\W_5 = \Hom(\O(-4) \oplus \O(-1), \O \oplus \O(1))$ and let $W_5 \subset \W_5$ be the subset
of morphisms $\f$ from proposition \ref{6.3}. The linear algebraic group
\[
G_5 = (\Aut(\O(-4) \oplus \O(-1)) \times \Aut(\O \oplus \O(1)))/\C^*
\]
acts on $\W_5$ by conjugation; $W_5$ is open and invariant in $\W_5$.
Let $X_5 \subset \M(6,1)$ be the locally closed subset of isomorphism classes
of cokernels of morphisms $\f \in W_5$.

\begin{prop}
\label{6.4}
There is a geometric quotient $W_5/G_5$, which is a smooth projective variety.
$W_5/G_5$ is isomorphic to the Hilbert flag scheme of sextic curves in $\P^2$ containing
zero-dimensional subschemes of length $2$.
\end{prop}

\begin{proof}
The argument is entirely analogous to the argument at 2.2.5 \cite{mult_five},
where we gave three constructions for the quotient.
\end{proof}

\begin{prop}
\label{6.5}
$W_5/G_5$ is isomorphic to $X_5$. In particular, $X_5$ is a smooth closed
subvariety of $\M(6,1)$ of codimension $8$.
\end{prop}

\begin{proof}
The canonical morphism $\r \colon W_5 \to X_5$ mapping $\f$ to the isomorphism class
of $\Coker(\f)$ determines a bijective morphism $\upsilon \colon W_5/G_5 \to X_5$.
Let $\HH$ be the Hilbert flag scheme of \ref{6.4}. Under the isomorphism $\HH \isom W_5/G_5$,
$\upsilon$ maps a point $(C,Z) \in \HH$ to $\J_Z(2)$, where $\J_Z \subset \O_C$ is the ideal sheaf
of $Z$ in $C$. Our aim is to show that $\upsilon^{-1}$ is also a morphism.
For this consider the good quotient $\pi \colon S \to X_5$ of 2.3.2 \cite{drezet-maican}.
In view of the universal property of a good quotient,
it is sufficient to show that $\upsilon^{-1} \circ \pi \colon S \to \HH$ is a morphism of varieties.
For this consider the $S$-flat family $\tilda{\F}_S$ on $\P^2 \times S$ defined at 2.3.3 \cite{drezet-maican}.
Let $p \colon \P^2 \times S \to S$ be the projection onto the second factor.
$\tilda{\F}_S$ satisfies the hypothesis of loc.cit., hence all higher direct image sheaves
$\operatorname{R}^j_{p_*}(\tilda{\F}_S \tensor \Om^{-i}(-i))$ are locally free on $S$
and, moreover, for any closed point $s \in S$, the restriction of the Beilinson tableau of $\tilda{\F}_S$
to a fibre $\P^2 \times \{ s\}$ is the Beilinson tableau (2.2.3) \cite{drezet-maican} for $\tilda{\F}_{S,s}$,
which we denote $\EE^1(\tilda{\F}_{S,s})$.
It remains to show that $(C,Z)$ can be obtained in a natural manner from $\EE^1(\tilda{\F}_{S,s})$,
if $\tilda{\F}_{S,s} \isom \J_Z(2) \subset \O_C(2)$.
In other words, given a sheaf $\F \isom \J_Z(2) \subset \O_C(2)$ in $X_5$,
we need to construct $(C,Z)$ starting from $\EE^1(\F)$ and performing algebraic operations.
By duality, given an extension
\[
0 \lra \O_C(2) \lra \G \lra \O_Z \lra 0
\]
as at \ref{6.2}, we need to obtain $(C,Z)$ in a natural manner from $\EE^1(\G)$.
The tableau for $\EE^1(\G)$ reads
\[
\xymatrix
{
4\O(-2) \ar[r]^-{\f_1} & 4\O(-1) \ar[r]^-{\f_2} & \O \\
3\O(-2) \ar[r]^-{\f_3} & 8\O(-1) \ar[r]^-{\f_4} & 6\O
}.
\]
We have $\Ker(\f_2) \isom \Om^1 \oplus \O(-1)$ because $\f_2$ is surjective.
Denote $\CC = \Ker(\f_2)/\Im(\f_1)$. Consider the Euler sequence on $\P^2$:
\[
0 \lra \O(-3) \lra 3\O(-2) \stackrel{\pi}{\lra} \Om^1 \lra 0.
\]
Clearly the corestriction $4\O(-2) \to \Om^1 \oplus \O(-1)$ of $\f_1$ factors through the morphism
\[
(\pi, id) \colon 3\O(-2) \oplus \O(-1) \to \Om^1 \oplus \O(-1).
\]
We arrive at an exact sequence 
\[
0 \lra \Ker(\f_1) \lra \O(-3) \oplus 4\O(-2) \stackrel{\eta}{\lra} 3\O(-2) \oplus \O(-1) \lra \CC \lra 0.
\]
As at 2.1.4 \cite{mult_five}, it can be shown that $\rank(\eta_{12})=3$. Canceling $3\O(-2)$ we get
an exact sequence
\[
0 \lra \Ker(\f_1) \lra \O(-3) \oplus \O(-2) \stackrel{\psi}{\lra} \O(-1) \lra \CC \lra 0.
\]
We cannot have $\CC \isom \O_Y(-1)$ for $Y \subset \P^2$ a line or a conic curve, otherwise $\CC$
would destabilise $\G$. It follows that $\CC$ is the structure sheaf of a zero-dimensional subscheme
in $\P^2$ of length $2$ and that $\Ker(\f_1) \isom \O(-4)$.
The exact sequence (2.2.5) \cite{drezet-maican} has the form
\[
0 \lra \O(-4) \stackrel{\f_5 \, }{\lra} \Coker(\f_4) \lra \G \lra \CC \lra 0.
\]
Denote $\G'= \Coker(\f_5)$.
From (2.2.4) \cite{drezet-maican} we get the resolution
\[
0 \lra 3\O(-2) \stackrel{\psi'}{\lra} \O(-4) \oplus 8\O(-1) \stackrel{\f'}{\lra} 6\O \lra \G' \lra 0,
\]
\[
\psi' = \left[
\ba{c}
0 \\ \f_3
\ea
\right], \qquad \f' = \left[
\ba{cc}
\f_5' & \f_4
\ea
\right].
\]
Here $\f_5'$ is a lift of $\f_5$.
We have $\h^0(\G')=6$, hence $\H^0(\G')= \H^0(\G)$.
The global sections of $\G$ generate $\O_C(2)$ and $\G'$ is generated by global sections.
Thus $\G'=\O_C(2)$.
The maximal minors of any matrix representing $\f'$ generate the ideal of $C$ because
the Fitting support of $\G'$ is $C$.
It is also clear that $\CC$ is isomorphic to $\O_Z$.

In conclusion, we have obtained the pair $(C,Z) \in \HH$ from $\EE^1(\G)$ by performing
algebraic operations.
\end{proof}

\begin{prop}
\label{6.6}
$X_5$ lies in the closure of $X_3$ and also in the closure of $X_4$.
\end{prop}

\begin{proof}
According to \ref{6.3}, the generic points in $X_5$ are stable-equivalence classes
of sheaves of the form $\O_C(2)(-P_1-P_2)$, where $C \subset \P_2$ is a smooth sextic curve
and $P_1, P_2$ are distinct points on $C$.
Choose points $P_3, P_4$ on $C$ such that $P_1, P_2, P_3$ are non-colinear,
$P_4$ is distinct from them and converges to $P_3$.
According to \ref{5.3}, the sheaf $\O_C(2)(-P_1 -P_2 -P_3 +P_4)$ gives a point in $X_3$.
This point converges to the stable-equivalence class of $\O_C(2)(-P_1 -P_2)$.
Thus $X_5 \subset \overline{X}_3$.

If $P_1$ and $P_2$ are generic enough, then the line they determine meets $C$ at four
other distinct points $Q_1, Q_2, Q_3, Q_4$. Choose points $P_i'$ on $C$ converging to $Q_i$,
$1 \le i \le 4$, such that no three of them are colinear.
According to \ref{5.8}, the sheaf $\O_C(1)(P_1' +P_2' +P_3' +P_4')$ gives a point in $X_4$.
This point converges to the stable-equivalence class of
$\O_C(1)(Q_1 +Q_2 +Q_3 +Q_4) \isom \O_C(2)(-P_1 -P_2)$.
Thus $X_5 \subset \overline{X}_4$.
\end{proof}

%%%%%%%%%%%%%%%%%%%%%%%%%%%%%%%%%%%%%%%%%% section 7

\section{The moduli space is a union of the strata}

\noi
In the final section we shall prove that there are no other sheaves giving points
in $\M(6,1)$ beside the sheaves we have discussed so far.

\begin{prop}
\label{7.1}
There are no sheaves $\F$ giving points in $\M(6,1)$ and satisfying the conditions
$\h^1(\F)= 1$ and $\h^0(\F(-1))=1$.
\end{prop}

\begin{proof}
By duality, we must show that there are no sheaves $\G$ giving points in $\M(6,5)$
and satisfying the conditions $\h^0(\G(-1))=1$, $\h^1(\G)=1$.
Consider a sheaf $\G$ on $\P^2$ with Hilbert polynomial $\PP_{\G}(t)= 6t+5$
and satisfying the above cohomological conditions.
Put $m = \h^1(\G \tensor \Om^1(1))$.
As in the proof of \ref{5.6}, the Beilinson free monad leads to the resolution
\[
0 \lra \O(-2) \lra 2\O(-2) \oplus (m+4)\O(-1) \stackrel{\f}{\lra} \Om^1 \oplus (m-3) \O(-1) \oplus 6\O \lra \G \lra 0.
\]
Here $\f_{12}=0$, $\f_{22} = 0$, hence $\G$ maps surjectively onto the cokernel 
$\CC$ of the morphism $2\O(-2) \to \Om^1 \oplus (m-3)\O(-1)$.
Thus $\CC$ has rank zero, forcing $m=3$. The Hilbert polynomial of $\CC$ is
$\PP(t) = \PP_{\Om^1}(t) - \PP_{2\O(-2)}(t)= t-1$, which shows that $\CC$ is a destabilising
quotient sheaf of $\G$. Thus $\G$ cannot give a point in $\M(6,5)$.
\end{proof}

\begin{prop}
\label{7.2}
There are no sheaves $\F$ giving points in $\M(6,1)$ and satisfying the cohomological conditions
\[
\h^0(\F(-1)) \le 1, \qquad \quad \h^1(\F) \ge 3, \qquad \quad \h^1(\F(1)) =0.
\]
\end{prop}

\begin{proof}
Assume that $\F$ gives a point in $\M(6,1)$ and satisfies the condition $\h^1(\F) \ge 3$.
Write $p = \h^1(\F)$, $m= \h^0(\F \tensor \Om^1(1))$. We will examine two cases, according
to the value of $\h^0(\F(-1))$.
Assume first that $\h^0(\F(-1))=0$. The Beilinson free monad for $\F$ reads
\[
0 \lra 5\O(-2) \oplus m\O(-1) \lra (m+4)\O(-1) \oplus (p+1)\O \stackrel{\psi}{\lra} p\O \lra 0,
\]
\[
\psi = \left[
\ba{cc}
\eta & 0
\ea
\right],
\]
and yields a resolution
\[
0 \lra 5\O(-2) \oplus m\O(-1) \stackrel{\f}{\lra} \Ker(\eta) \oplus (p+1)\O \lra \F \lra 0
\]
in which $\f_{12}=0$. From the injectivity of $\f$ we see that $m+4-p = \rank(\Ker(\eta)) \le 5$.
Thus
\[
\h^0(\F(1)) = 3(p+1) + \h^0(\Ker(\eta)(1)) - m \ge 2p+2 \ge 8
\]
forcing $\h^1(\F(1)) > 0$.

\medskip

\noi
Assume next that $\h^0(\F(-1))=1$. The Beilinson free monad for the dual sheaf $\G= \F^\D(1)$ reads
\[
0 \lra p\O(-2) \lra (p+1)\O(-2) \oplus (m+4)\O(-1) \lra m\O(-1) \oplus 6\O \lra \O \lra 0
\]
and yields the resolution
\[
0 \lra p\O(-2) \lra (p+1)\O(-2) \oplus (m+4)\O(-1) \lra \Om^1 \oplus (m-3)\O(-1) \oplus 6\O \lra \G \lra 0,
\]
hence the resolution
\begin{multline*}
0 \lra p\O(-2) \lra \O(-3) \oplus (p+1)\O(-2) \oplus (m+4) \O(-1) \lra \\
3 \O(-2) \oplus (m-3) \O(-1) \oplus 6\O \lra \G \lra 0.
\end{multline*}
As in the proof of 2.1.4 \cite{mult_five}, the map $(p+1)\O(-2) \to 3\O(-2)$ has rank $3$,
hence we may cancel $3\O(-2)$ to get the exact sequence
\begin{multline*}
0 \lra p\O(-2) \lra \O(-3) \oplus (p-2)\O(-2) \oplus (m+4) \O(-1) \lra \\
(m-3) \O(-1) \oplus 6\O \lra \G \lra 0.
\end{multline*}
Since $\G$ maps surjectively onto the cokernel $\CC$ of the morphism
\[
\O(-3) \oplus (p-2)\O(-2) \lra (m-3)\O(-1),
\]
we have $m-3 \le p-1$.
Moreover, if $m-3 = p-1$, then $\CC$ has Hilbert polynomial $\PP(t)=pt-1$, hence $\CC$ destabilises
$\G$. Thus we have the inequality $m \le p+1$.
According to \cite{maican-duality}, lemma 3, we may dualise the above resolution to get a monad for
$\F$ of the form
\[
0 \lra 6\O(-2) \oplus (m-3)\O(-1) \lra (m+4)\O(-1) \oplus (p-2)\O \oplus \O(1) \stackrel{\psi}{\lra} p\O \lra 0,
\]
\[
\psi = \left[
\ba{ccc}
\eta & 0 & 0
\ea
\right].
\]
This yields the resolution
\[
0 \lra 6\O(-2) \oplus (m-3)\O(-1) \lra \Ker(\eta) \oplus (p-2)\O \oplus \O(1) \lra \F \lra 0.
\]
Thus
\[
\h^0(\F(1)) = 3(p-2) + 6 + \h^0(\Ker(\eta)(1)) -(m-3) \ge 2p+2 \ge 8,
\]
forcing $\h^1(\F(1)) > 0$.
\end{proof}

\begin{prop}
\label{7.3}
There are no sheaves $\F$ giving points in $\M(6,1)$ and satisfying the cohomological
condition $\h^0(\F(-1))=2$.
\end{prop}

\begin{proof}
Assume that there is $\F$ as in the proposition. 
Write $p = \h^1(\F)$, $m= \h^0(\F \tensor \Om^1(1))$.
The Beilinson free monad for $\F$ reads
\[
0 \lra 2\O(-2) \stackrel{\xi}{\lra} 7\O(-2) \oplus m\O(-1) \lra (m+4)\O(-1) \oplus (p+1)\O \lra p\O \lra 0.
\]
As in the proof of 3.2.5 \cite{mult_five}, we can show that $m \ge 6$ and that $\xi$ is equivalent
to the morphism represented by the matrix
\[
\left[
\ba{ccccccccc}
0 & \cdots & 0 & X & Y & Z & 0 & 0 & 0 \\
0 & \cdots & 0 & 0 & 0 & 0 & X & Y & Z
\ea
\right]^\T.
\]
We recall that the argument is based on the fact that there is no non-zero morphism
$\O_L(1) \to \F$ for any line $L \subset \P^2$.
According to \cite{maican-duality}, lemma 3, taking duals of the locally free sheaves in the above
monad yields a monad for the dual of $\F$. This monad gives the following resolution for
the sheaf $\G = \F^\D(1)$:
\[
0 \! \lra p\O(-2) \lra (p+1)\O(-2) \oplus (m+4)\O(-1) \lra 2\Om^1 \oplus (m-6)\O(-1) \oplus 7\O \lra \G \lra \! 0.
\]
This further leads to the resolution
\begin{multline*}
0 \lra p\O(-2) \lra 2\O(-3) \oplus (p+1)\O(-2) \oplus (m+4) \O(-1) \lra \\
6\O(-2) \oplus (m-6) \O(-1) \oplus 7\O \lra \G \lra 0.
\end{multline*}
Since $\G$ maps surjectively onto the cokernel $\CC$ of the morphism
\[
2\O(-3) \oplus (p+1)\O(-2) \lra 6\O(-2) \oplus (m-6)\O(-1),
\]
we have $m\le p+3$.
Moreover, if $m = p+3$, then $\CC$ has Hilbert polynomial $\PP(t)=(p-1)t-2$, hence $\CC$ destabilises
$\G$. We deduce that $m \le p+2$. As $m \ge 6$, we have $p \ge 4$.
As above, the dual monad for $\F$ takes the form
\[
0 \lra 7\O(-2) \oplus (m-6)\O(-1) \oplus 6\O \lra (m+4)\O(-1) \oplus (p+1)\O \oplus 2\O(1) \stackrel{\psi}{\lra} p\O \lra 0,
\]
\[
\psi = \left[
\ba{ccc}
\eta & 0 & 0
\ea
\right].
\]
Thus
\[
\h^0(\F(1)) = 3(p+1)+ 12+ \h^0(\Ker(\eta)(1))- (m-6)- 18 \ge 2p+1 \ge 9,
\]
forcing $\h^1(\F(1)) \ge 2$. According to \ref{6.3}, this is impossible.
\end{proof}

\begin{prop}
\label{7.4}
Let $\F$ be a sheaf giving a point in $\M(6,1)$. Then $\h^0(\F(-1))=0 $ or $1$.
\end{prop}

\begin{proof}
Assume that $\F$ gives a point in $\M(6,1)$ and satisfies the condition $\h^0(\F(-1))>0$.
As in the proof of 2.1.3 \cite{drezet-maican}, there is an injective morphism $\O_C \to \F(-1)$
for a curve $C \subset \P^2$. From the semi-stability of $\F$ we see that $C$ has degree $5$ or $6$.
In the first case $\F(-1)/\O_C$ has Hilbert polynomial $\PP(t)=t$ and has no zero-dimensional torsion.
Indeed, the pull-back in $\F(-1)$ of any non-zero subsheaf of $\F(-1)/\O_C$ supported on finitely
many points would destabilise $\F(-1)$. We deduce that $\F(-1)/\O_C$ is isomorphic to
$\O_L(-1)$ for a line $L \subset \P^2$, hence $\h^0(\F(-1))=1$.

Assume now that $C$ is a sextic curve.
The quotient sheaf $\CC= \F(-1)/\O_C$ has support of dimension zero and length $4$.
Assume that $\h^0(\F(-1))> 1$. Then, in view of \ref{7.3}, we have $\h^0(\F(-1)) \ge 3$.
We claim that there is a global section $s$ of $\F(-1)$ such that its image in $\CC$
generates a subsheaf isomorphic to $\O_Z$, where $Z \subset \P^2$ is a zero-dimensional
scheme of length $1$, $2$ or $3$.
Indeed, as $\h^0(\O_C)=1$ and $\h^0(\F(-1))$ is assumed to be at least $3$,
there are global sections $s_1$ and $s_2$ of $\F(-1)$ such that their images in $\CC$
are linearly independent.
It is easy to see that there exists a subsheaf $\CC' \subset \CC$ of length $3$.
Choose $c_1, c_2 \in \C$, not both zero, such that the image of $c_1 s_1 + c_2 s_2$
under the composite map $\F(-1) \to \CC \to \CC/\CC'$ is zero.
Then $s=c_1 s_1 + c_2 s_2$ satisfies our requirements.

Let $\F' \subset \F(-1)$ be the preimage of $\O_Z$. Assume first that $Z$ is not contained in a line,
so, in particular, it has length $3$. According to \cite{modules-alternatives}, proposition 4.5, we have
a resolution
\[
0 \lra 2\O(-3) \lra 3\O(-2) \lra \O \lra \O_Z \lra 0.
\]
We apply the horseshoe lemma to the extension
\[
0 \lra \O_C \lra \F' \lra \O_Z \lra 0,
\]
to the standard resolution of $\O_C$ and to the resolution of $\O_Z$ from above.
We obtain the exact sequence
\[
0 \lra 2\O(-3) \lra \O(-6) \oplus 3\O(-2) \lra 2\O \lra \F' \lra 0.
\]
As the morphism $2\O(-3) \to \O(-6)$ in the above complex is zero and as $\Ext^1(\O_Z, \O)=0$,
we can show, as in the proof of 2.3.2 \cite{mult_five},
that $\O_Z$ is a direct summand of $\F'$.
This is absurd, by hypothesis $\F(-1)$ has no zero-dimensional torsion.
The same argument applies if $Z$ is contained in a line and has length $3$,
except that this time we use the resolution
\[
0 \lra \O(-4) \lra \O(-3) \oplus \O(-1) \lra \O \lra \O_Z \lra 0.
\]
The cases when $Z$ has length $1$ or $2$ are analogous.
We conclude that $\h^0(\F(-1))=1$.
\end{proof}

\end{document}